\documentclass[10.5pt,a4paper]{article}

\usepackage{graphicx,latexsym,euscript,makeidx,color,bm}
\usepackage{amsmath,amsfonts,amssymb,amsthm,thmtools,mathrsfs,enumerate}
\usepackage[colorlinks,linkcolor=blue,anchorcolor=green,citecolor=red]{hyperref}
\usepackage{geometry}
   \def\cD{{\cal D}}  
\def\dbE{\mathbb{E}}     
\def\dbF{\mathbb{F}}   \def\cF{{\cal F}}  
     
\def\dbH{\mathbb{H}}

\def\dbP{\mathbb{P}}     
   \def\cQ{{\cal Q}}  
\def\dbR{\mathbb{R}} \def\sR{\mathscr{R}}    
\def\dbS{\mathbb{S}}     
     
   \def\cU{{\cal U}}

\def\ss{\smallskip}      \def\lt{\left}       \def\hb{\hbox}
\def\ms{\medskip}        \def\rt{\right}      \def\ae{\hbox{\rm a.e.}}
        \def\lan{\langle}    
\def\ds{\displaystyle}   \def\ran{\rangle}    \def\tr{\hbox{\rm tr$\,$}}
      \def\llan{\lt\lan}   
\def\no{\noindent}       \def\rran{\rt\ran}   
\def\ns{\noalign{\ss}}

\def\rf{\eqref}            \def\hp{\hphantom}
\def\deq{\triangleq}     \def\({\Big (}       \def\nn{\nonumber}
\def\les{\leqslant}      \def\){\Big )}       
\def\ges{\geqslant}      \def\[{\Big[}        
\def\ti{\tilde}          \def\]{\Big]}        
      \def\q{\quad}        
         \def\qq{\qquad}      \def\1n{\negthinspace}
\def\cd{\cdot}           \def\2n{\1n\1n}      \def\3n{\1n\1n\1n}

              \def\Om{\Omega}  
         \def\D{\Delta}   \def\d{\delta}   \def\F{\Phi}     
\def\z{\zeta}         \def\Th{\Theta}      \def\si{\sigma}
\def\e{\varepsilon}     \def\l{\lambda}        
         \def\f{\varphi}  \def\i{\infty}   

\def\ba{\begin{array}}                \def\ea{\end{array}}
\def\bel{\begin{equation}\label}      \def\ee{\end{equation}}

\newtheorem{theorem}{Theorem}[section]
\newtheorem{definition}[theorem]{Definition}
\newtheorem{proposition}[theorem]{Proposition}
\newtheorem{corollary}[theorem]{Corollary}
\newtheorem{lemma}[theorem]{Lemma}
\newtheorem{remark}[theorem]{Remark}
\newtheorem{example}[theorem]{Example}
\newtheorem{taggedassumptionx}{}
\newenvironment{taggedassumption}[1]
 {\renewcommand\thetaggedassumptionx{#1}\taggedassumptionx}
 {\endtaggedassumptionx}
\makeatletter
   
   \@addtoreset{equation}{section}
\makeatother
  \allowdisplaybreaks[4]
\begin{document}

\title{\bf Mean-Field Stochastic Linear-Quadratic Optimal
Control Problems: Weak Closed-Loop Solvability}
\author{Jingrui Sun\thanks{Department of Mathematics, Southern University of Science and Technology,
                           Shenzhen, Guangdong 518055, China ({\tt sunjr@sustech.edu.cn} ).} \and
       Hanxiao Wang\thanks{Corresponding author. School of Mathematical Sciences, Fudan University,
                           Shanghai 200433, China (Email: {\tt hxwang14@} {\tt fudan.edu.cn}).}}
%\date{}
\maketitle

{\no\bf Abstract.}
This paper is concerned with mean-field stochastic linear-quadratic (MF-SLQ, for short)
optimal control problems with deterministic coefficients.
The notion of weak closed-loop optimal strategy is introduced.
It is shown that the  open-loop solvability is equivalent to the
existence of a weak closed-loop optimal strategy.
Moreover, when open-loop optimal controls exist, there is at least one of them admitting
a state feedback representation, which is the outcome of a weak closed-loop optimal strategy.
Finally, an example is presented to illustrate the procedure for finding weak closed-loop
optimal strategies.

\ms

{\no\bf Keywords.}
mean-field stochastic differential equation, linear-quadratic optimal control, open-loop solvability,
weak closed-loop solvability, state feedback.

\ms

{\no\bf AMS subject classifications.} 93E20, 49N10, 49N35.

\section{Introduction}\label{Sec:Introduction}

Throughout this paper, $(\Om,\cF,\dbP)$ is a given complete probability space
on which a standard one-dimensional Brownian motion $W=\{W(t);0\les t<\i\}$ is defined.
The augmented natural filtration of $W$ is denoted by $\dbF=\{\cF_t\}_{t\ges0}$. Let
\begin{align*}
L_{\cF_{t}}^2(\Om;\dbR^n) &= \big\{\xi:\Om\to\dbR^n~| ~\xi \hbox{ is $\cF_{t}$-measurable with } \dbE|\xi|^2<\i\big\},  \\
                      \cD &= \big\{(t,\xi)~|~t\in[0,T),\,\xi\in L^2_{\cF_t}(\Om;\dbR^n)\big\}.
\end{align*}
For any {\it initial pair} $(t,\xi)\in\cD$, consider the following controlled linear mean-filed
stochastic differential equation (MF-SDE, for short) on the finite horizon $[t,T]$:
\bel{state}\left\{\begin{aligned}
   dX(s) &=\big\{A(s)X(s)+ \bar A(s)\dbE[X(s)] +B(s)u(s)+\bar B(s)\dbE[u(s)]+ b(s)\big\}ds \\
         &\hp{=\ } +\big\{C(s)X(s)+\bar C(s)\dbE[X(s)] + D(s)u(s)+\bar D(s)\dbE[u(s)]+\si(s)\big\}dW(s),\\
     X(t)&= \xi,
\end{aligned}\right.\ee
where $A,\bar A,\,C,\bar C:[0,T]\to\dbR^{n\times n}$, $B,\bar B,\, D,\bar D:[0,T]\to\dbR^{n\times m}$
are given deterministic functions, called the {\it coefficients} of the {\it state equation} \rf{state};
$b,\si:[0,T]\times\Om\to\dbR^n$ are $\dbF$-progressively measurable processes, called the {\it nonhomogeneous terms}.
In the above, the solution $X$ of \rf{state} is called a {\it state process},
and $u:[t,T]\times\Om\to \dbR^m$ is called a {\it control process}, which is an element of the following space:
$$\cU[t,T] = \bigg\{u:[t,T]\times\Om\to \dbR^m~\big|~u\hb{~is $\dbF$-progressively measurable and~}
\dbE\int^T_t|u(s)|^2ds<\i\bigg\}. $$
According to the classical results of mean-field SDEs (see \cite{Yong 2017}, for example),
under some mild conditions, for any $(t,\xi)\in\cD$ and $u(\cd)\in\cU[t,T]$,
equation \rf{state} admits a unique  solution $X(\cd)\equiv X(\cd\,;t,\xi,u(\cd))$.
To measure the performance of the control $u(\cd)$, we introduce the following quadratic {\it cost functional}:
\begin{align}\label{cost}
J(t,\xi;u(\cd))
&= \dbE\big\{\lan GX(T),X(T)\ran + 2\lan g,X(T)\ran +\lan\bar G\dbE[X(T)],\dbE[X(T)]\ran + 2\lan \bar g,\dbE[X(T)]\ran\nn\\
&\hp{=\ } +\int_t^T\bigg[\llan\begin{pmatrix}Q (s)& S(s)^\top \\ S(s) & R(s)\end{pmatrix}
                                 \begin{pmatrix}X(s) \\ u(s)\end{pmatrix},
                                 \begin{pmatrix}X(s) \\ u(s)\end{pmatrix}\rran +2\llan\begin{pmatrix}q(s) \\ \rho(s)\end{pmatrix},
                   \begin{pmatrix}X(s) \\ u(s)\end{pmatrix}\rran\bigg]ds\nn\\
&\hp{=\ } +\int_t^T\bigg[\llan\begin{pmatrix}\bar Q (s)& \bar S(s)^\top \\ \bar S(s) & \bar R(s)\end{pmatrix}
                                 \begin{pmatrix}\dbE[X(s)] \\ \dbE[u(s)]\end{pmatrix},
                                 \begin{pmatrix}\dbE[X(s)] \\ \dbE[u(s)]\end{pmatrix}\rran\nn\\
&\hp{=\ }\qq\qq +2\llan\begin{pmatrix}\bar q(s) \\ \bar \rho(s)\end{pmatrix},
                   \begin{pmatrix}\dbE[X(s)] \\ \dbE[u(s)]\end{pmatrix}\rran\bigg]ds\Bigg\},
\end{align}
where $G,\bar G\in\dbR^{n\times n}$ are symmetric constant matrices;
$g$ is an $\cF_T$-measurable $\dbR^n$-valued random vector and $\bar g$ is a (deterministic) $\dbR^n$-valued vector;
$Q,\bar Q:[0,T]\to\dbR^{n\times n}$, $S,\bar S:[0,T]\to\dbR^{m\times n}$, and $R,\bar R:[0,T]\to\dbR^{m\times m}$
are deterministic functions with $Q,\bar Q$ and $R,\bar R$ being symmetric;
$q:[0,T]\times\Om\to\dbR^n$, $\rho:[0,T]\times\Om\to\dbR^m$ are $\dbF$-progressively measurable processes;
and $\bar q:[0,T]\to\dbR^n$, $\bar\rho:[0,T]\to\dbR^m$ are deterministic functions.
In the above, $M^\top$ stands for the transpose of a matrix $M$.
With the state equation \rf{state} and the  cost functional \rf{cost},
the {\it mean-field stochastic linear-quadratic (LQ, for short) optimal control problem} can be stated as follows:

\ms

{\bf Problem (MF-SLQ).} For any given initial pair $(t,\xi)\in\cD$,
find a control $\bar u(\cd)\in \cU[t,T]$ such that
\bel{inf J}J(t,\xi;\bar u(\cd))=\inf_{u(\cd)\in\cU[t,T]} J(t,\xi;u(\cd))\deq V(t,\xi).\ee

\ss

Any $\bar u(\cd)\in\cU[t,T]$ satisfying \rf{inf J} is called an ({\it open-loop}) {\it optimal control}
of Problem (MF-SLQ) for the initial pair $(t,\xi)$;
the corresponding state process $\bar X(\cd)\equiv X(\cd\,;t,\xi,\bar u(\cd))$ is called an ({\it open-loop})
{\it optimal state process};
$(\bar X(\cd),\bar u(\cd))$ is called an ({\it open-loop}) {\it optimal pair};
and  $V(\cd\,,\cd):\cD\to\dbR$ is called the {\it value function} of Problem (MF-SLQ).
Note that in the special case when $b(\cd),\si(\cd),g,\bar g,q(\cd),\bar q(\cd),\rho(\cd)$ and $\bar \rho(\cd)$ are absent,
the state equation \rf{state} and the cost functional \rf{cost} reduce to
\bel{state0}\left\{\begin{aligned}
   dX(s) &=\big\{A(s)X(s)+ \bar A(s)\dbE[X(s)] +B(s)u(s)+\bar B(s)\dbE[u(s)]\big\}ds\\
         &\hp{=\ } +\big\{C(s)X(s)+\bar C(s)\dbE[X(s)]  + D(s)u(s)+\bar D(s)\dbE[u(s)]\big\}dW(s),\\
     X(t)&= \xi,
\end{aligned}\right.\ee
and
\begin{align}\label{cost0}
J^0(t,\xi;u(\cd))
&= \dbE\big\{\lan GX(T),X(T)\ran +\lan\bar G\dbE[X(T)],\dbE[X(T)]\ran\nn\\
&\hp{=\ } +\int_t^T\llan\begin{pmatrix}Q (s)& S(s)^\top \\ S(s) & R(s)\end{pmatrix}
                                 \begin{pmatrix}X(s) \\ u(s)\end{pmatrix},
                                 \begin{pmatrix}X(s) \\ u(s)\end{pmatrix}\rran ds\nn\\
&\hp{=\ } +\int_t^T\llan\begin{pmatrix}\bar Q (s)& \bar S(s)^\top \\ \bar S(s) & \bar R(s)\end{pmatrix}
                                 \begin{pmatrix}\dbE[X(s)] \\ \dbE[u(s)]\end{pmatrix},
                                 \begin{pmatrix}\dbE[X(s)] \\ \dbE[u(s)]\end{pmatrix}\rran ds\Bigg\},
\end{align}
respectively. In this case, we denote the corresponding mean-field stochastic LQ problem
and its value function by Problem (MF-SLQ)$^0$ and $V^0(\cd\,,\cd)$, respectively.

\ms

When the mean-field part vanishes, Problem (MF-SLQ) becomes a classical stochastic
LQ optimal control problem, which has been well studied by many researchers;
see, for example, \cite{Wonham 1968,Bismut 1976,Chen-Li-Zhou 1998,Yong-Zhou 1999,
Chen-Zhou 2000,Rami-Moore-Zhou 2001,Chen-Yong 2001,Sun-Li-Yong 2016,Wang 2019}
and the references cited therein.
LQ optimal control problems for MF-SDEs over a finite horizon were first studied by Yong \cite{Yong 2013},
and were later extended to the infinite horizon by Huang, Li, and Yong \cite{Huang 2014}.
Recently, based on the idea of \cite{Sun-Yong 2014,Sun-Li-Yong 2016}, Sun \cite{Sun 2017}
and Li, Sun, and Yong \cite{Li-Sun-Yong 2016} investigated the open-loop and closed-loop
solvabilities for Problem (MF-SLQ) and found that these two types of solvabilities
are essentially different.
More precisely, they showed in \cite{Li-Sun-Yong 2016} that the closed-loop solvability of Problem (MF-SLQ)
is equivalent to the existence of a {\it regular} solution to the following generalized Riccati equation
(GRE, for short):
\bel{GRE1}\left\{\begin{aligned}
& \dot P+PA+A^\top P+C^\top PC+Q \\
& \hp{\dot P} -(PB+C^\top PD+S^\top)(R+D^\top PD)^\dag(B^\top P+D^\top PC+ S)=0, \\
& \dot{\Pi}+\Pi(A+\bar A)+(A+\bar A)^\top\Pi+Q+\bar Q+(C+\bar C)^\top P(C+\bar C)\\
& \hp{\dot{\Pi}} -\big[\Pi(B+\bar B)+(C+\bar C)^\top P(D+\bar D)+(S+\bar S)^\top\big]\big[R+\bar R+(D+\bar D)^\top P(D+\bar D)\big]^\dag\\
& \hp{\dot \Pi} \times\big[(B+\bar B)^\top\Pi+(D+\bar D)^\top P(C+\bar C)+(S+\bar S)\big]=0, \\
&P(T)=G,\q \Pi(T)=G+\bar G,
\end{aligned}\right.\ee
(where $M^\dag$ denotes the Moore--Penrose pseudoinverse of a matrix $M$ and the argument $s$ is suppressed),
and that the closed-loop solvability implies the open-loop solvability of Problem (MF-SLQ), but not vice-versa.
The advantage of existence of a closed-loop optimal strategy is that a state feedback optimal control, which is
the outcome of some closed-loop optimal strategy, can be explicitly constructed in terms of the solution to \rf{GRE1}.
However, as just mentioned, Problem (MF-SLQ) might be merely open-loop solvable,
in which case solving the GRE \rf{GRE1} will fail to produce a state feedback optimal control.
To see this, let us consider the following example.

\begin{example}\label{ex-1.1}\rm
Consider the one-dimensional state equation
$$\left\{\begin{aligned}
   dX(s) &= \big\{X(s)+\dbE[X(s)]+u(s)+\dbE[u(s)]\big\}ds +\big \{X(s)+\dbE[X(s)]\big\}dW(s), \q s\in[t,1],\\
    X(t) &= \xi,
\end{aligned}\right.$$
and the nonnegative cost functional
$$ J(t,x;u(\cd))=|\dbE [X(1)]|^2. $$
In this example, the associated GRE reads
\bel{example-Ric}\left\{\begin{aligned}
   & \dot{P}(s)+3P(s)=0,\q \dot\Pi(s)+4\Pi(s)+4P(s)=0, \q s\in[t,1],\\
    &P(1)=0,\qq \Pi(1)=1.
\end{aligned}\right. \ee
It is easily to verify that the unique solution of \rf{example-Ric} is $(P(s),\Pi(s))\equiv(0,e^{4-4s})$,
which, however, is not regular according to the definition in \cite{Li-Sun-Yong 2016}.
If we use the usual Riccati equation approach to construct the state feedback optimal control $u^*(\cd)$,
then $u^*(\cd)$ should be given by the following (noting that $R(\cd)=0$, $\bar R(\cd)=0$, $D(\cd)=0$, $\bar D(\cd)=0$ and $0^\dag=0$):
\begin{align*}
 u^*(s)& \deq -\big[R(s)+D(s)^\top P(s)D(s)\big]^\dag\big[B(s)^\top P(s)+D(s)^\top P(s)C(s)+ S(s)\big]\{X(s)-\dbE[X(s)]\}\\
 &\q-\big[R(s)+\bar R(s)+(D(s)+\bar D(s))^\top P(s)(D(s)+\bar D(s))\big]^\dag\\
 &\qq\times\big[(B(s)+\bar B(s))^\top \Pi(s)+(D(s)+\bar D(s))^\top P(s)(C(s)+\bar C(s))+ S(s)+\bar S(s)\big]\dbE[X(s)]\\
 &\equiv 0.
\end{align*}
Such a control is {\it not} open-loop optimal if the initial state $\xi$ satisfies $\dbE[\xi]\neq0$.
Indeed, by the variation of constants formula, the expectation of the state process $X^*(\cd)$
corresponding to $(t,\xi)$ and $u^*(\cd)$ is given by
$$\dbE[X^*(s)]=e^{2s-2t}\dbE[\xi],\q s\in[t,1].$$
Thus,
$$J(t,x;u^*(\cd))=|\dbE[X^*(1)]|^2=e^{4-4t}|\dbE[\xi]|^2>0.$$
On the other hand, let $\bar u(\cd)$ be the control defined by
$$\bar u(s)\equiv {\dbE[\xi]\over 2(t-1)}e^{2s-2t}, \q s\in[t,1]. $$
By the variation of constants formula, the expectation of the state process
$\bar X(\cd)$ corresponding to $(t,\xi)$ and $\bar u(\cd)$ is given by
\begin{align*}
\dbE[\bar X(s)]&=e^{2s-2t}\dbE[\xi]+2\int_t^s e^{2s-2r}\dbE[u(r)]dr\\
&=e^{2s-2t}\dbE[\xi]\big(1+{s-t\over t-1}\big),\q s\in[t,1],
\end{align*}
which satisfies $\dbE[\bar X(1)]=0$. Hence,
$$ J(t,x;\bar u(\cd)) =|\dbE[\bar X(1)]|^2 =0 <J(t,x;u^*(\cd)). $$
Since the cost functional is nonnegative, $\bar u(\cd)$ is open-loop optimal for the initial pair $(t,\xi)$, but $u^*(\cd)$ is not.
\end{example}

Now some questions arise naturally:
{\it When Problem (MF-SLQ) is merely open-loop solvable, does state feedback optimal control exists?
If yes, how can we find such an optimal control?}
The objective of this paper is to answer these questions.
We shall first provide an alternative characterization of the open-loop solvability of Problem (MF-SLQ)
using the perturbation approach introduced by Sun, Li, and Yong \cite{Sun-Li-Yong 2016}.
This characterization, which avoids the subsequence extraction, is a refinement
of \cite[Theorem 3.2]{Sun 2017}.
Then we generalize the notion of weak closed-loop strategies, which is first introduced
by Wang, Sun, and Yong \cite{Wang 2019} for classical stochastic LQ problems, to the mean-field case.
We shall show that the existence of a weak closed-loop optimal strategy is equivalent to
the existence of an open-loop optimal control,
and that as long as Problem (MF-SLQ) is open-loop solvable, a state feedback optimal control
always exists and can be represented as the outcome of a weak closed-loop optimal strategy.
Moreover, our constructive proof provides a procedure for finding weak closed-loop optimal strategies.

\ms

The rest of the paper is organized as follows.
In \autoref{Sec:Preliminaries}, we collect some preliminary results and introduce
a few elementary notions for Problem (MF-SLQ).
\autoref{Sec:3} is devoted to the study of open-loop solvability by a perturbation method.
In \autoref{Sec:4}, we show how to obtain a weak closed-loop optimal strategy
and establish the equivalence between open-loop and weak closed-loop solvabilities.
An example is presented in \autoref{Sec:Example} to illustrate the results we obtained.

\section{Preliminaries}\label{Sec:Preliminaries}
Let $\dbR^{n\times m}$ be the Euclidean space consisting of $n\times m$ real matrices, endowed with the Frobenius inner product $\lan M,N\ran\deq\tr[M^\top N]$, where $M^\top$ and $\tr(M)$ stand for the transpose of a matrix $M$ and  the trace of $M$, respectively.
The  Frobenius norm of a matrix $M$ is denoted by $|M|$.
We shall denote by $I_n$ the identity matrix of size $n$ and by  $\dbS^n$ the subspace of $\dbR^{n\times n}$ consisting of symmetric matrices.
%We will use $K>0$ to represent a generic constant which could be different from line to line.
Let $T>0$ be a fixed time horizon.
For any $t\in[0,T)$ and  Euclidean space $\dbH$ (which could be $\dbR^n$, $\dbR^{n\times m}$, $\dbS^n$, etc.), we introduce the following spaces of functions and processes:
\begin{align*}
C([t,T];\dbH):
   &\hb{~~the space of $\dbH$-valued, continuous functions on $[t,T]$}; \\
L^p(t,T;\dbH):
   &\hb{~~the space of $\dbH$-valued functions that are $p$th $(1\les p\les\i)$} \\
   &\hb{~~power Lebesgue integrable on $[t,T]$};\\
L^2_{\cF_T}(\Om;\dbH):
   &\hb{~~the space of $\cF_T$-measurable, $\dbH$-valued random variables $\xi$} \\
   &\hb{~~with $\dbE|\xi|^2<\i$};\\
L^2_\dbF(\Om;L^1(t,T;\dbH)):
   &\hb{~~the space of $\dbF$-progressively measurable, $\dbH$-valued processes} \\
   &\hb{~~$\f:[t,T]\times\Om\to\dbH$ with $\dbE\big[\ds\int_t^T|\f(s)|ds\big]^2<\i$};\\
L_\dbF^2(t,T;\dbH):
   &\hb{~~the space of $\dbF$-progressively measurable, $\dbH$-valued processes} \\
   &\hb{~~$\f:[t,T]\times\Om\to\dbH$ with $\ds\dbE\int_t^T|\f(s)|^2ds<\i$};\\
L_\dbF^2(\Om;C([t,T];\dbH)):
   &\hb{~~the space of $\dbF$-adapted, continuous, $\dbH$-valued processes} \\
   &\hb{~~$\f:[t,T]\times\Om\to\dbH$ with $\dbE\big[\sup_{s\in[t,T]}|\f(s)|^2\big]<\i$}.
\end{align*}
For $M,N\in\dbS^n$, we use the notation $M\ges N$ (respectively, $M>N$) to indicate that $M-N$ is positive
semi-definite (respectively, positive definite).
Further, for any $\dbS^n$-valued measurable function $F$ on $[t,T]$, we denote
$$\left\{\begin{aligned}
  &F\ges 0\q \Longleftrightarrow\q F(s)\ges 0,\q \ae~s\in[t,T],\\
 &F> 0\q \Longleftrightarrow\q F(s)> 0,\q \ae~s\in[t,T],\\
 &F\gg 0\q \Longleftrightarrow\q F(s)\ges \d I_n,\q \ae~s\in[t,T],~\hbox{for some } \d>0.
\end{aligned}\right. $$

For the state equation \rf{state} and cost functional \rf{cost}, we introduce the following assumptions:
\begin{taggedassumption}{(H1)}\label{ass:A1}
The coefficients and the nonhomogeneous terms of the state equation \rf{state} satisfy
$$\left\{\begin{aligned}
   A(\cd),\bar A(\cd) &\in L^1(0,T;\dbR^{n\times n}),     &&&   B(\cd),\bar B(\cd) &\in L^2(0,T;\dbR^{n\times m}),&&& b(\cd) &\in L^2_\dbF(\Om;L^1(0,T;\dbR^n)),\\
   C(\cd),\bar C(\cd) &\in L^2(0,T;\dbR^{n\times n}),     &&&   D(\cd),\bar D(\cd) &\in L^\i(0,T;\dbR^{n\times m}),&&& \si(\cd) &\in L_\dbF^2(0,T;\dbR^n).
\end{aligned}\right. $$
\end{taggedassumption}
\begin{taggedassumption}{(H2)}\label{ass:A2}
The weighting coefficients in the cost functional \rf{cost} satisfy
$$\left\{\2n\ba{ll}
\ns\ds Q(\cd),\bar Q(\cd)\in L^1(0,T;\dbS^n),\q S(\cd),\bar S(\cd)\in L^2(0,T;\dbR^{m\times n}),
\q R(\cd),\bar R(\cd)\in L^\infty(0,T;\dbS^m), \\
\ns\ds q(\cd)\in L^2_\dbF(\Om;L^1(0,T;\dbR^n)),\q\rho(\cd)\in L_\dbF^2(0,T;\dbR^m),\q g\in L^2_{\cF_T}(\Om;\dbR^n),\\
\ns\ds \bar q(\cd)\in L^1(0,T;\dbR^n),\q\bar\rho(\cd)\in L^2(0,T;\dbR^m),\q\bar g\in\dbR^n,\q G,\bar G\in\dbS^n.\ea\right.$$

\end{taggedassumption}
Under the assumption \ref{ass:A1}, we have the following well-posedness of the state equation, whose proof is standard and can be found in \cite[Proposition 2.1]{Yong 2017}.

\begin{lemma}\label{lmm:well-posedness-SDE}
Let {\rm\ref{ass:A1}} hold. Then for any initial pair $(t,\xi)\in\cD$ and control $u(\cd)\in\cU[t,T]$,
the state equation \rf{state} admits a unique  solution $X(\cd)\equiv X(\cd\,;t,x,u(\cd))$.
Moreover, there exists a constant $K>0$, independent of $(t,\xi)$ and $u(\cd)$, such that
$$ \dbE\lt[\sup_{t\les s\les T}|X(s)|^2\rt]
\les K\dbE\lt[|\xi|^2+\lt(\int_t^T|b(s)|ds\rt)^2 + \int_t^T|\si(s)|^2ds + \int^T_t|u(s)|^2ds\rt].$$
\end{lemma}

Suppose that \ref{ass:A1} holds. Then according to \autoref{lmm:well-posedness-SDE}, for any initial pair $(t,\xi)\in\cD$ and control $u(\cd)\in\cU[t,T]$, equation \rf{state} admits a unique  solution $X(\cd)\equiv X(\cd\,;t,x,u(\cd))\in L_\dbF^2(\Om;C([t,T];\dbR^n))$.
Hence, under the additional assumption \ref{ass:A2}, the random variables on the right-hand side of  \rf{cost} are integrable and Problem (MF-SLQ) is well-posed. We now recall the following notions of mean-filed stochastic LQ problems.
\begin{definition}\rm
Problem (MF-SLQ) is said to be
\begin{enumerate}[(i)]
\item ({\it uniquely}) {\it open-loop solvable at $(t,\xi)\in\cD$} if
      there exists a (unique) $\bar u(\cd)\equiv \bar u(\cd\,;t,\xi)\in \cU[t,T]$ (depending on $(t,\xi)$) such that
      \bel{open-optimal} J(t,\xi;\bar u(\cd))\les J(t,\xi;u(\cd)),\q\forall u(\cd)\in\cU[t,T]. \ee
      Such a $\bar u(\cd)$ is called an {\it open-loop optimal control of Problem (MF-SLQ) for  $(t,\xi)$}.
\item ({\it uniquely}) {\it open-loop solvable at $t\in[0,T)$} if for the given $t$ and any $\xi\in L_{\cF_t}(\Om;\dbR^n)$, it is (uniquely) open-loop solvable at  $(t,\xi)$.
%\item ({\it uniquely}) {\it open-loop solvable at $t$} if it is (uniquely) open-loop solvable at $(t,x)$ for all $x\in\dbR^n$;
      %
\item ({\it uniquely}) {\it open-loop solvable} if it is (uniquely) open-loop solvable at any initial pair $(t,\xi)\in\cD$.
\end{enumerate}
\end{definition}

\begin{definition}\rm
Let $\Th,\,\bar\Th:[t,T]\to\dbR^{m\times n}$ be two deterministic functions and $v:[t,T]\times\Om\to\dbR^m$
be an $\dbF$-progressively measurable process.
\begin{enumerate}[(i)]
\item We call $(\Th(\cd),\bar \Th(\cd),v(\cd))$ a {\it closed-loop strategy} on $[t,T]$ if $\Th(\cd),\bar \Th(\cd)\in L^2(t,T;\dbR^{m\times n})$
      and $v(\cd)\in L_\dbF^2(t,T;\dbR^m)$; that is,
      $$\int_t^T|\Th(s)|^2ds<\i, \q\int_t^T|\bar\Th(s)|^2ds<\i, \q \dbE\int_t^T|v(s)|^2ds<\i.$$
      The set of all closed-loop strategies $(\Th(\cd),\bar\Th(\cd),v(\cd))$ on $[t,T]$ is denoted by $\cQ[t,T]$.

\item A closed-loop strategy $(\Th^*(\cd),\bar\Th^*(\cd),v^*(\cd))\in\cQ[t,T]$
      is said to be {\it optimal} on $[t,T]$ if
      \begin{equation}\label{closed-optimal*}
      \begin{aligned}
        J(t,\xi;\Th^*(\cd)X^*(\cd)+\bar\Th^*(\cd)\dbE[X^*(\cd)]+v^*(\cd))\les J(t,\xi;\Th(\cd)X(\cd)+\bar\Th(\cd)\dbE[X(\cd)]+v(\cd)),\\
       \forall \xi\in L^2_{\cF_t}(\Om;\dbR^n),~\forall(\Th(\cd),\bar\Th(\cd),v(\cd))\in\cQ[t,T],
       \end{aligned}
      \end{equation}
      where $X^*(\cd)$ is the solution to the {\it closed-loop system} under $(\Th^*(\cd),\bar\Th^*(\cd),v^*(\cd))$
      (with the argument $s$ being suppressed in the coefficients and non-homogeneous terms):
      \bel{closed-syst*}\left\{\begin{aligned}
         dX^*(s) &=\big\{AX^*(s)+\bar A\dbE[X^*(s)]+B\big\{\Th^*X^*(s)+\bar\Th^*\dbE[X^*(s)]+v^*\big\} \\
                 &\hp{=\ }\qq +\bar B\dbE\big\{\Th^*X(s)+\bar\Th^*\dbE[X^*(s)]+v^*\big\}+b\big\}ds\\
                  &\hp{=\ }+\big\{CX^*(s)+\bar C\dbE[X^*(s)]+D\big\{\Th^*X^*(s)+\bar\Th^*\dbE[X^*(s)]+v^*\big\} \\
                 &\hp{=\ }\qq +\bar D\dbE\big\{\Th^*X(s)+\bar\Th^*\dbE[X^*(s)]+v^*\big\}+\si\big\}dW(s),\\
          X^*(t) &=\xi,
      \end{aligned}\right.\ee
      and $X(\cd)$ is the solution to the following closed-loop system under $(\Th(\cd),\bar\Th(\cd),v(\cd))$:
      \bel{closed-syst}\left\{\begin{aligned}
         dX(s) &=\big\{AX(s)+\bar A\dbE[X(s)]+B\big\{\Th X(s)+\bar\Th\dbE[X(s)]+v\big\} \\
                 &\hp{=\ }\qq +\bar B\dbE\big\{\Th X(s)+\bar\Th\dbE[X(s)]+v\big\}+b\big\}ds\\
                  &\hp{=\ }+\big\{CX(s)+\bar C\dbE[X(s)]+D\big\{\Th X(s)+\bar\Th\dbE[X(s)]+v\big\} \\
                 &\hp{=\ }\qq +\bar D\dbE\big\{\Th X(s)+\bar\Th\dbE[X(s)]+v\big\}+\si\big\}dW(s),\\
          X(t) &=\xi.
      \end{aligned}\right.\ee

\item If for any $t\in[0,T)$, a closed-loop optimal strategy (uniquely) exists on $[t,T]$,
      Problem (MF-SLQ) is said to be ({\it uniquely}) {\it closed-loop solvable}.
\end{enumerate}
\end{definition}
Motivated by \autoref{ex-1.1} and  \cite{Wang 2019}, we next introduce the  notion of weak closed-loop strategies for mean-field stochastic LQ problems.
\begin{definition}\rm\label{def-wcloop}
Let $\Th,\bar\Th:(t,T)\to\dbR^{m\times n}$ be two locally square-integrable deterministic functions and $v:(t,T)\times\Om\to\dbR^m$ be a locally square-integrable $\dbF$-progressively measurable process; that is, $\Th(\cd),\,\bar\Th(\cd)$ and $v(\cd)$ are such that for any $t^\prime,T^\prime\in(t,T)$,
$$\int_{t^\prime}^{T^\prime}|\Th(s)|^2ds<\i, \q\int_{t^\prime}^{T^\prime}|\bar\Th(s)|^2ds<\i, \q \dbE\int_{t^\prime}^{T^\prime}|v(s)|^2ds<\i.$$
\begin{enumerate}[(i)]
\item We call $(\Th(\cd),\bar\Th(\cd),v(\cd))$ a {\it weak closed-loop strategy} on $(t,T)$ if for any initial state $\xi\in L^2_{\cF_t}(\Om;\dbR^n)$,
      the outcome $u(\cd)\equiv\Th(\cd)X(\cd)+\bar\Th(\cd)\dbE[X(\cd)]+v(\cd)$ of $(\Th(\cd),\bar\Th(\cd),v(\cd))$ belongs to $\cU[t,T]$, where $X(\cd)$ is the solution to the {\it weak closed-loop system}:
      \bel{weak-syst}\left\{\begin{aligned}
         dX(s) &=\big\{AX(s)+\bar A\dbE[X(s)]+B\big\{\Th X(s)+\bar\Th\dbE[X(s)]+v\big\} \\
                 &\hp{=\ }\qq +\bar B\dbE\big\{\Th X(s)+\bar\Th\dbE[X(s)]+v\big\}+b\big\}ds\\
                  &\hp{=\ }+\big\{CX(s)+\bar C\dbE[X(s)]+D\big\{\Th X(s)+\bar\Th\dbE[X(s)]+v\big\} \\
                 &\hp{=\ }\qq +\bar D\dbE\big\{\Th X(s)+\bar\Th\dbE[X(s)]+v\big\}+\si\big\}dW(s),\\
          X(t) &=\xi.
      \end{aligned}\right.\ee
      The set of all weak closed-loop strategies is denoted by $\cQ_w[t,T]$.
\item A weak closed-loop strategy $(\Th^*(\cd),\bar\Th^*(\cd),v^*(\cd))$ is said to be {\it optimal} on $(t,T)$ if
      \bel{weak-closed-optimal*}\ba{ll}
       \ns\ds J(t,\xi;\Th^*(\cd)X^*(\cd)+\bar\Th^*(\cd)\dbE[X^*(\cd)]+v^*(\cd))\les J(t,\xi;\Th(\cd)X(\cd)+\bar\Th(\cd)\dbE[X(\cd)]+v(\cd)),\\
       \ns\ds\qq\qq\qq\qq\qq\qq\forall \xi\in L_{\cF_t}(\Om;\dbR^n),~\forall(\Th(\cd),\bar\Th(\cd),v(\cd))\in\cQ_w[t,T],\ea \ee
      where $X(\cd)$ is the solution of the weak closed-loop system \rf{weak-syst}, and $X^*(\cd)$ is the solution to the weak closed-loop system \rf{weak-syst} corresponding to $(t,\xi)$ and $(\Th^*(\cd),\bar\Th^*(\cd),v^*(\cd))$.
\item If for any $t\in[0,T)$, a weak closed-loop optimal strategy (uniquely) exists on $(t,T)$,
      we say Problem (SLQ) is ({\it uniquely}) {\it weakly closed-loop solvable}.
\end{enumerate}
\end{definition}
%\begin{remark}\rm
%Note that the weak closed-loop strategy $(\Th(\cd),\bar\Th(\cd),v(\cd))$ of Problem (MF-SLQ) belongs to a different space from the weak closed-loop strategy of %classical stochastic LQ problem introduced in \cite[Definition 2.5]{Wang 2019}.
%\end{remark}
Similar to the case of classical stochastic LQ problem (see \cite{Wang 2019}, for example), we have the following equivalence: A (weak) closed-loop strategy $(\Th^*(\cd),\bar\Th^*(\cd),v^*(\cd))\in\cQ_w[t,T]$ is (weakly) closed-loop optimal
on $(t,T)$ if and only if
\bel{weak-closed-optimal1*}J(t,\xi;\Th^*(\cd)X^*(\cd)+\bar\Th^*(\cd)\dbE[X^*(\cd)]+v^*(\cd))\les J(t,\xi;u(\cd)),\q\forall \xi\in L^2_{\cF_{t}}(\Om;\dbR^n),~\forall u(\cd)\in\cU[t,T]. \ee

In the sequel, we shall use the following result, which is concerned with the open-loop and closed-loop solvabilities of Problem (MF-SLQ) and
whose proof  can be found in Sun \cite{Sun 2017}.

\begin{theorem}\label{thm:SLQ-ccloop-kehua}
Let {\rm\ref{ass:A1}} and {\rm\ref{ass:A2}} hold.
\begin{enumerate}[\rm(i)]
\item Suppose Problem {\rm(MF-SLQ)} is open-loop solvable. Then $J^0(0,0;u(\cd))\ges 0$ for all $u(\cd)\in \cU[0,T]$.
\item Suppose that there exists a constant $\l>0$ such that
      $$ J^0(0,0;u(\cd))\ges\l\dbE\int_0^T|u(s)|^2ds, \q\forall u(\cd)\in \cU[0,T]. $$
      Then the Riccati equation \rf{GRE1}
      %
    %\bel{GRE}\left\{\begin{aligned}
     % & \dot P+PA+A^\top P+C^\top PC+Q \\
     % & \hp{\dot P} -(PB+C^\top PD+S^\top)(R+D^\top PD)^\dag(B^\top P+D^\top PC+ S)=0, \\
       %
      %  &R+D^\top PD\ges0,\q\ae,\\
      %  &\sR\big(B^\top P+D^\top PD\big)\subseteq\sR\big(R+D^\top PD\big),\q\ae,\\
      %& P(T)=G
      % \end{aligned}\right.\ee
      %
      admits a unique solution $(P(\cd),\Pi(\cd))\in C([0,T];\dbS^n)\times C([0,T];\dbS^n)$ such that
      $$\Sigma\equiv R+D^\top P D\gg 0,\q\bar\Sigma\equiv R+\bar R+(D+\bar D)^\top P(D+\bar D)\gg 0.$$
      %
      %and the corresponding Riccati equation
      %\bel{DGRE}\left\{\begin{aligned}
      %& \dot{\Pi}+\Pi(A+\bar A)+(A+\bar A)^\top\Pi+Q+\bar Q+(C+\bar C)^\top P(C+\bar C)\\
      %&  -[\Pi(B+\bar B)+(C+\bar C)^\top P(D+\bar D)+(S+\bar S)^\top][R+\bar R+(D+\bar D)^\top P(D+\bar D)]^\dag\\
       %& \hp{\dot \Pi} \cdot[(B+\bar B)^\top\Pi+(D+\bar D)^\top P(C+\bar C)+(S+\bar S)]=0, \\
       %
      %  &R+D^\top PD\ges0,\q\ae,\\
      %  &\sR\big(B^\top P+D^\top PD\big)\subseteq\sR\big(R+D^\top PD\big),\q\ae,\\
      %& \Pi(T)=G+\bar G
      % \end{aligned}\right.\ee
      %
      Problem {\rm(MF-SLQ)} is uniquely closed-loop solvable and hence  uniquely open-loop solvable.
      The unique closed-loop optimal strategy $(\Th^*(\cd),\bar\Th^*(\cd),v^*(\cd))$ is given by
      $$\left\{\begin{aligned}
        \Th^* &= -\Sigma^{-1}\big(B^\top P+D^\top PC+S\big),\\
         \bar\Th^* &= \ti\Th^*-\Th^*,\q v^* = \varphi-\dbE[\varphi]+\bar\varphi,
      \end{aligned}\right.$$
      where
      $$\left\{\begin{aligned}
       \ti\Th^* &= -\bar\Sigma^{-1}[(B+\bar B)^\top\Pi+(D+\bar D)^\top P(C+\bar C)+S+\bar S],\\
      \varphi&= -\Sigma^{-1}\big(B^\top\eta+D^\top\z+D^\top P\si+\rho\big),\\
      \bar\varphi&=-\bar\Sigma^{-1}\big\{(B+\bar B)^\top\bar\eta+(D+\bar D)^\top (\dbE[\zeta]+P\dbE[\si])+\dbE[\rho]+\bar\rho\big\},
      \end{aligned}\right.$$
      with $(\eta(\cd),\zeta(\cd))$ and $\bar\eta(\cd)$ being the (adapted) solutions to the BSDE
      \bel{eta-beta}\left\{\begin{aligned}
         d\eta(s) &=-\big[(A+B\Th^*)^\top\eta(s) +(C+D\Th^*)^\top\z(s) +(C+D\Th^*)^\top P\si\\
                  &\hp{=-\big[} +\Th^{*\top}\rho +Pb +q\big]ds +\z(s) dW(s), \q s\in[t,T],\\
          \eta(T) &=g,
      \end{aligned}\right.\ee
      and ordinary differential equation
        \bel{bar-eta}\left\{\begin{aligned}
         &\dot{\bar\eta}(s)+\big[(A+\bar A)+(B+\bar B)\ti\Th^*)\big]^\top\bar\eta(s) +\ti\Th^{*\top}\big\{(D+\bar D)^\top\big(P\dbE[\si]+\dbE[\zeta]\big)+\dbE[\rho]+\bar\rho\big\}\\
                  &\hp{=-\big[} +(C+\bar C)^\top\big(P\dbE[\si]+\dbE[\zeta]\big)+\dbE[q]+\bar q+\Pi\dbE[b]=0, \q s\in[t,T],\\
          &\bar\eta(T) =\dbE[g]+\bar g,
      \end{aligned}\right.\ee
      respectively. The unique open-loop optimal control of Problem {\rm(MF-SLQ)} for the initial pair $(t,\xi)$ is given by
      \begin{align*}
       u^*(s) &= \Th^*(s)X^*(s)+\bar\Th^*(s)\dbE[X^*(s)]+v^*(s)\\
       &= \Th^*(s)\big\{X^*(s)-\dbE[X^*(s)]\big\}+\ti\Th^*(s)\dbE[X^*(s)]+v^*(s), \q s\in[t,T],
      \end{align*}
      where $X^*(\cd)$ is the solution to the corresponding closed-loop system \rf{closed-syst*}.
\end{enumerate}
\end{theorem}

\section{A Perturbation Approach to Open-Loop Solvability}\label{Sec:3}

In this section, we shall study the open-loop solvability of Problem (MF-SLQ) under the following convexity condition:
\bel{cost-convex} J^0(0,0;u(\cd))\ges0, \q\forall u(\cd)\in\cU[0,T].\ee
By \autoref{thm:SLQ-ccloop-kehua} (i), the above condition is necessary for the open-loop solvability of Problem (MF-SLQ).
Condition \rf{cost-convex} means that the mapping $u(\cd)\mapsto J^0(0,0;u(\cd))$ is convex.
In fact, under the condition \rf{cost-convex}, one can  prove the mapping $u(\cd)\mapsto J(t,\xi;u(\cd))$ is convex for any initial pair $(t,\xi)\in\cD$.
We point out that the condition \rf{cost-convex} is not sufficient for the open-loop solvability of Problem (MF-SLQ)$^0$,
one counterexample could be found in \cite[Example 5.2]{Sun-Li-Yong 2016}.

\ms
For any $\e>0$, let us consider the LQ problem of minimizing the perturbed cost functional
\begin{eqnarray}\label{cost-e}
J_\e(t,\xi;u(\cd))&\3n\deq\3n& J(t,\xi;u(\cd))+\e\dbE\int_t^T|u(s)|^2ds\nn\\
&\3n=\3n& \dbE\big\{\lan GX(T),X(T)\ran + 2\lan g,X(T)\ran +\lan\bar G\dbE[X(T)],\dbE[X(T)]\ran + 2\lan \bar g,\dbE[X(T)]\ran\nn\\
         &\3n~\3n& +\int_t^T\bigg[\llan\begin{pmatrix}Q (s)& S(s)^\top \\ S(s) & R(s)+\e I_m\end{pmatrix}
                                 \begin{pmatrix}X(s) \\ u(s)\end{pmatrix},
                                 \begin{pmatrix}X(s) \\ u(s)\end{pmatrix}\rran +2\llan\begin{pmatrix}q(s) \\ \rho(s)\end{pmatrix},
                   \begin{pmatrix}X(s) \\ u(s)\end{pmatrix}\rran\bigg]ds\nn\\
&\3n~\3n& +\int_t^T\bigg[\llan\begin{pmatrix}\bar Q (s)& \bar S(s)^\top \\ \bar S(s) & \bar R(s)\end{pmatrix}
                                 \begin{pmatrix}\dbE[X(s)] \\ \dbE[u(s)]\end{pmatrix},
                                 \begin{pmatrix}\dbE[X(s)] \\ \dbE[u(s)]\end{pmatrix}\rran\nn\\
&\3n~\3n&       \qq\qq              +2\llan\begin{pmatrix}\bar q(s) \\ \bar \rho(s)\end{pmatrix},
                   \begin{pmatrix}\dbE[X(s)] \\ \dbE[u(s)]\end{pmatrix}\rran\bigg]ds\Bigg\}
\end{eqnarray}
subject to the state equation \rf{state}.
We denote this perturbed LQ problem by Problem (MF-SLQ)$_\e$ and its value function by $V_\e(\cd\,,\cd)$.
Then the cost functional of the homogeneous LQ problem associated with Problem (MF-SLQ)$_\e$ is give by
$$J^0_\e(t,\xi;u(\cd))=J^0(t,\xi;u(\cd))+\e\dbE\int_t^T|u(s)|^2ds.$$
Note that, by \rf{cost-convex}, we have
$$ J^0_{\e}(0,0;u(\cd))\ges\e\dbE\int_0^T|u(s)|^2ds, \q\forall u\in\mathcal{U}[0,T].$$
According to \autoref{thm:SLQ-ccloop-kehua} (ii), the above implies that the Riccati equation
\bel{Ric-e}\left\{\begin{aligned}
  & \dot P_\e + P_\e A + A^\top P_\e + C^\top P_\e C + Q \\
  & \hp{\dot P_\e} -(P_\e B+C^\top P_\e D+S^\top)(R+\e I_m+D^\top P_\e D)^{-1}(B^\top P_\e +D^\top P_\e C+S)=0, \\
   & \dot{\Pi}_\e+\Pi_\e(A+\bar A)+(A+\bar A)^\top\Pi_\e+Q+\bar Q+(C+\bar C)^\top P_\e(C+\bar C)\\
      &  -[\Pi_\e(B+\bar B)+(C+\bar C)^\top P_\e(D+\bar D)+(S+\bar S)^\top][R+\e I_m+\bar R+(D+\bar D)^\top P_\e(D+\bar D)]^{-1}\\
       & \hp{\dot \Pi} \cdot[(B+\bar B)^\top\Pi_\e+(D+\bar D)^\top P_\e(C+\bar C)+(S+\bar S)]=0, \\
  & P_\e(T)=G,\qq \Pi_\e(T)=G+\bar G
\end{aligned}\right.\ee
associated with Problem (MF-SLQ)$_\e$ admits a unique solution $(P_\e(\cd),\Psi_\e(\cd))\in C([0,T];\dbS^n)\times C([0,T];\dbS^n)$ such that
$$\Sigma_\e\equiv R+\e I_m+D^\top P D\gg 0,\q\bar\Sigma_\e\equiv R+\e I_m+\bar R+(D+\bar D)^\top P_\e(D+\bar D)\gg 0.$$
%
%Let $\Pi_\e(\cd)$ be the unique solution to the following Riccati equation
%\bel{DGRE-e}\left\{\begin{aligned}
 %     & \dot{\Pi}_\e+\Pi_\e(A+\bar A)+(A+\bar A)^\top\Pi_\e+Q+\bar Q+(C+\bar C)^\top P_\e(C+\bar C)\\
 %     &  -[\Pi_\e(B+\bar B)+(C+\bar C)^\top P_\e(D+\bar D)+(S+\bar S)^\top][R+\e I_m+\bar R+(D+\bar D)^\top P_\e(D+\bar D)]^{-1}\\
 %      & \hp{\dot \Pi} \cdot[(B+\bar B)^\top\Pi_\e+(D+\bar D)^\top P_\e(C+\bar C)+(S+\bar S)]=0, \\
  %    & \Pi_\e(T)=G+\bar G.
%\end{aligned}\right.\ee
Define
\begin{align}
  \label{Th-e} &\Th_\e = -\Sigma_\e^{-1}(B^\top P_\e+D^\top P_\e C+S), \\
 \label{bar-Th-e}     &\bar\Th_\e = \ti\Th_\e-\Th_\e,\\
    \label{v-e}  & v_\e = \varphi_\e-\dbE[\varphi_\e]+\bar\varphi_\e,
\end{align}
where
\bel{def-ti-Th-Th-v-e}\left\{\begin{aligned}
&\ti\Th_\e =-\bar\Sigma_\e^{-1}[(B+\bar B)^\top\Pi_\e+(D+\bar D)^\top P_\e(C+\bar C)+S+\bar S],\\
%    %
      &\varphi_\e= -\Sigma_\e^{-1}\big(B^\top\eta_\e+D^\top\z_\e+D^\top P_\e\si+\rho\big),\\
      &\bar\varphi_\e=-\bar\Sigma_\e^{-1}\big\{(B+\bar B)^\top\bar\eta_\e+(D+\bar D)^\top (\dbE[\zeta_\e]+P_\e\dbE[\si])+\dbE[\rho]+\bar\rho\big\},\\
      \end{aligned}\right.\ee
with $(\eta_\e(\cd),\z_\e(\cd))$ being the unique adapted solution to the BSDE
\bel{eta-zeta-e}\left\{\begin{aligned}
  d\eta_\e(s) &=-\big[(A+B\Th_\e)^\top\eta_\e(s)+(C+D\Th_\e)^\top\z_\e(s) + (C+D\Th_\e)^\top P_\e\si \\
              &\hp{=-\big[} +\Th_\e^\top\rho + P_\e b +q\big]ds + \z_\e(s) dW(s), \q s\in[0,T],\\
   \eta_\e(T) &=g,
\end{aligned}\right.\ee
and $\bar\eta_\e(\cd)$ being the solution to the following ODE
        \bel{bar-eta-e}\left\{\begin{aligned}
         &\dot{\bar\eta}_\e(s)+\big[(A+\bar A)+(B+\bar B)\ti\Th_\e)\big]^\top\bar\eta_\e(s) +\ti\Th^{\top}_\e\big\{(D+\bar D)^\top\big(P_\e\dbE[\si]+\dbE[\zeta_\e]\big)+\dbE[\rho]+\bar\rho\big\}\\
                  &\hp{=-\big[} +(C+\bar C)^\top\big(P\dbE[\si]+\dbE[\zeta_\e]\big)+\dbE[q]+\bar q+\Pi_\e\dbE[b]=0, \q s\in[t,T],\\
          &\bar\eta_\e(T) =\dbE[g]+\bar g.
      \end{aligned}\right.\ee
Let $X_\e(\cd)$ be the unique solution to the closed-loop system
      \bel{closed-loop-syst-e}\left\{\begin{aligned}
         dX_\e(s) &=\big\{AX_\e(s)+\bar A\dbE[X_\e(s)]+B\big\{\Th_\e X_\e(s)+\bar\Th_\e\dbE[X_\e(s)]+v_\e\big\} \\
                 &\hp{=\ }\qq +\bar B\dbE\big\{\Th_\e X_\e(s)+\bar\Th_\e\dbE[X_\e(s)]+v_\e\big\}+b\big\}ds\\
                  &\hp{=\ }+\big\{CX_\e(s)+\bar C\dbE[X_\e(s)]+D\big\{\Th_\e X_\e(s)+\bar\Th_\e\dbE[X_\e(s)]+v_\e\big\} \\
                 &\hp{=\ }\qq +\bar D\dbE\big\{\Th_\e X_\e(s)+\bar\Th_\e\dbE[X_\e(s)]+v_\e\big\}+\si\big\}dW(s),\\
          X_\e(t) &=\xi,
      \end{aligned}\right.\ee
then the unique open-loop optimal control of Problem (MF-SLQ)$_\e$ for the initial pair $(t,\xi)$ is given by
\bel{u-e} u_\e(s) = \Th_\e(s)X_\e(s)+\bar\Th_e\dbE[X_\e(s)] + v_\e(s), \q s\in[t,T]. \ee

\ss

Now we are ready to state the main result of this section, which provides a characterization of the open-loop
solvability of Problem (MF-SLQ) in terms of the family $\{u_\e(\cd)\}_{\e>0}$ defined by \rf{u-e}.

\begin{theorem}\label{thm:SLQ-oloop-kehua}
Let {\rm\ref{ass:A1}}, {\rm\ref{ass:A2}} and \rf{cost-convex} hold.
For any given initial pair $(t,\xi)\in\cD$, let $u_\e(\cd)$ be defined by \rf{u-e}, which is the outcome of the closed-loop optimal strategy $(\Th_\e(\cd),\bar\Th_\e(\cd),v_\e(\cd))$ of Problem {\rm(MF-SLQ)}$_\e$. Then the following statements are equivalent:
\begin{enumerate}[\rm(i)]
\item Problem {\rm(MF-SLQ)} is open-loop solvable at $(t,\xi)$;
\item the family $\{u_\e(\cd)\}_{\e>0}$ is bounded in the Hilbert space $L^2_\dbF(t,T;\dbR^m)$, i.e.,
      $$\sup_{\e>0}\,\dbE\int_t^T|u_\e(s)|^2ds <\i;$$
\item the family $\{u_\e(\cd)\}_{\e>0}$ is convergent strongly in $L^2_\dbF(t,T;\dbR^m)$ as $\e\to0$.
\end{enumerate}
Whenever {\rm(i)}, {\rm(ii)}, or {\rm(iii)} is satisifed, the family $\{u_\e(\cd)\}_{\e> 0}$ converges strongly to an open-loop
optimal control of Problem {\rm(MF-SLQ)} for the initial pair $(t,\xi)$ as $\e\to0$.
\end{theorem}

To prove \autoref{thm:SLQ-oloop-kehua}, we first present the following lemma, which is borrowed from  \cite[Theorem 3.2]{Sun 2017}.

\begin{lemma}\label{lemma-V}
Let {\rm\ref{ass:A1}} and {\rm\ref{ass:A2}} hold. Then for any initial pair $(t,\xi)\in\cD$,
\bel{Ve-goto-V} \lim_{\e\downarrow0}V_\e(t,\xi)=V(t,\xi). \ee
\end{lemma}
%
%\begin{proof}
%Let $(t,\xi)\in\cD$ be fixed. For any $\e>0$ and any $u(\cd)\in\mathcal{U}[t,T]$, we have
%
%$$ J_\e(t,\xi;u(\cd))=J(t,\xi;u(\cd))+\e\dbE\int_t^T|u(s)|^2ds \ges J(t,\xi;u(\cd))\ges V(t,\xi).$$
%Taking the infimum over all $u(\cd)\in\mathcal{U}[t,T]$ on the left-hand side gives
%
%\bel{Ve<V-case1} V_\e(t,\xi) \ges V(t,\xi). \ee
%
%On the other hand, if $V(t,\xi)$ is finite, then for any $\d>0$, we can find a $u^\d(\cd)\in\mathcal{U}[t,T]$, independent of $\e>0$, such that $J(t,\xi;u^\d(\cd))\les V(t,\xi)+\d$. It follows that
%
%$$ V_\e(t,\xi)\les J(t,\xi;u^\d(\cd))+\e\dbE\int_t^T|u^\d(s)|^2ds\les V(t,\xi)+\d+\e\dbE\int_t^T|u^\d(s)|^2ds. $$
%
%Letting $\e\to0$, we obtain
%
%\bel{Ve<V-case2} V_\e(t,\xi) \les V(t,\xi)+\d. \ee
%
%Since $\d>0$ is arbitrary, we obtain \rf{Ve-goto-V} by combining \rf{Ve<V-case1} and \rf{Ve<V-case2}.
%A similar argument applies to the case when $V(t,\xi)=-\i$.
%\end{proof}

\no\it Proof of \autoref{thm:SLQ-oloop-kehua}. \rm
We begin by proving the implication  (i) $\Rightarrow$ (ii).
Let $v^*(\cd)$ be an open-loop optimal control of Problem (MF-SLQ) for the initial pair $(t,\xi)$. Then for any $\e>0$,
$$
\left\{
\begin{aligned}
  V_\e(t,\xi) &\les J(t,\xi;v^*(\cd))+\e\dbE\int_t^T|v^*(s)|^2ds = V(t,\xi) + \e\dbE\int_t^T|v^*(s)|^2ds,\\
  V_\e(t,\xi) & = J(t,\xi;u_\e(\cd)) + \e\dbE\int_t^T|u_\e(s)|^2ds \ges V(t,\xi) +\e\dbE\int_t^T|u_\e(s)|^2ds, \end{aligned}\right.$$
which yields
\bel{bound-of-ue} \dbE\int_t^T|u_\e(s)|^2ds \les {V_\e(t,\xi)-V(t,\xi)\over\e} \les \dbE\int_t^T|v^*(s)|^2ds. \ee
This shows that $\{u_\e(\cd)\}_{\e>0}$ is bounded in $L^2_\dbF(t,T;\dbR^m)$.

\ms

We next show that (ii) $\Rightarrow$ (i).
Since $\{u_\e(\cd)\}_{\e>0}$ is bounded in the Hilbert space  $L^2_\dbF(t,T;\dbR^m)$, there is a subsequence
$\{\e_k\}^\i_{k=1}$ of $\{\e\}_{\e>0}$ with $\lim_{k\to\i}\e_k=0$ such that $\{u_{\e_k}(\cd)\}$
converges weakly to some $u^*(\cd)\in L^2_\dbF(t,T;\dbR^m)$.
Note that the mapping $u(\cd)\mapsto J(t,\xi;u(\cd))$ is sequentially weakly lower semicontinuous because it is continuous and convex.
By \autoref{lemma-V} , we have
\begin{align*}
J(t,\xi;u^*(\cd))&\les \liminf_{k\to\i} J(t,\xi;u_{\e_k}(\cd)) \\
          &= \liminf_{k\to\i}\lt[ V_{\e_k}(t,\xi)-\e_k\dbE\int_t^T|u_{\e_k}(s)|^2ds \rt]= V(t,\xi),
\end{align*}
which implies that $u^*(\cd)$ is an open-loop optimal control of Problem (MF-SLQ) for $(t,\xi)$.

\ms

The implication (iii) $\Rightarrow$ (ii) is trivially true.

\ms

Finally, we prove the implication (ii) $\Rightarrow$ (iii). The proof is divided into two steps.

\ms

{\it Step 1:  The family $\{u_\e(\cd)\}_{\e>0}$ converges weakly to an open-loop optimal control of Problem {\rm(MF-SLQ)} for the initial pair $(t,\xi)$ as $\e\to0$.}

\ms

To verify this, it suffices to show that every weakly convergent subsequence of $\{u_\e(\cd)\}_{\e>0}$ has the same weak limit.
Let $u_i^*(\cd)$; $i=1,2$, be the weak limits of two different weakly convergent subsequences $\{u_{i,\e_k}(\cd)\}_{k=1}^\infty$ $(i=1,2)$ of $\{u_\e(\cd)\}_{\e>0}$. Then similar the proof of (ii) $\Rightarrow$ (i), it is clear to see that both $u_1^*(\cd)$ and $u^*_2(\cd)$ are optimal for $(t,\xi)$.
Thus, by the convexity of the mapping $u(\cd)\mapsto J(t,\xi;u(\cd))$, we have
$$ J\lt(t,\xi;{u_1^*(\cd)+u_2^*(\cd)\over 2}\rt)\les {1\over2}J(t,\xi;u_1^*(\cd)) +{1\over2}J(t,\xi;u_2^*(\cd))= V(t,\xi). $$
This shows that ${u_1^*(\cd)+u_2^*(\cd)\over 2}$ is also an optimal control of Problem (MF-SLQ) at $(t,\xi)$. Then we can repeat the argument employed in the proof of (i) $\Rightarrow$ (ii), replacing $v^*(\cd)$ by ${u_1^*(\cd)+u_2^*(\cd)\over 2}$, to obtain (see \rf{bound-of-ue})
$$ \dbE\int_t^T|u_{i,\e_k}(s)|^2ds \les \dbE\int_t^T\lt|{u_1^*(s)+u_2^*(s)\over 2}\rt|^2ds, \q i=1,2. $$
Taking inferior limits on the both sides of the above inequality then yields
$$ \dbE\int_t^T|u_i^*(s)|^2ds \les \dbE\int_t^T\lt|{u_1^*(s)+u_2^*(s)\over 2}\rt|^2ds, \q i=1,2. $$
Adding the above two inequalities and then multiplying by $2$, we get
$$ 2\lt[\dbE\int_t^T|u_1^*(s)|^2ds+\dbE\int_t^T|u_2^*(s)|^2ds\rt] \les \dbE\int_t^T|u_1^*(s)+u_2^*(s)|^2ds. $$
By shifting the integral on the right-hand side to the left-hand side, we have
$$ \dbE\int_t^T|u_1^*(s)-u_2^*(s)|^2ds\les 0. $$
It follows that $u_1^*(\cd)=u_2^*(\cd)$, which establishes the claim.

\ms

{\it Step 2: The family $\{u_\e(\cd)\}_{\e>0}$ converges strongly as $\e\to0$. }

\ms

According to Step 1, the family $\{u_\e(\cd)\}_{\e>0}$ converges weakly to an open-loop optimal control $u^*(\cd)$ of Problem {\rm(MF-SLQ)} for $(t,\xi)$ as $\e\to0$.
Similar to \rf{bound-of-ue} (with $u^*(\cd)$ replacing $v^*(\cd)$), we get
\bel{t-3-5}\dbE\int_t^T|u_\e(s)|^2ds \les \dbE\int_t^T|u^*(s)|^2ds, \q\forall\e>0.\ee
On the other hand, since $u^*(\cd)$ is the weak limit of $\{u_\e(\cd)\}_{\e>0}$, we have
\bel{t-3-3-6} \dbE\int_t^T|u^*(s)|^2ds \les \liminf_{\e\to0} \dbE\int_t^T|u_\e(s)|^2ds.\ee
The above, together with \rf{t-3-5}, yields that
 $$\lim_{\e\to0}\dbE\int_t^T|u_\e(s)|^2ds=\dbE\int_t^T|u^*(s)|^2ds.$$
Thus, recalling that $\{u_\e(\cd)\}_{\e>0}$ converges weakly to $u^*(\cd)$, we have
\begin{align*}
  &\lim_{\e\to0}\dbE\int_t^T|u_\e(s)-u^*(s)|^2ds \\
  &\q =\lim_{\e\to0}\lt[\dbE\int_t^T|u_\e(s)|^2ds + \dbE\int_t^T|u^*(s)|^2ds -2\,\dbE\int_t^T\lan u^*(s),u_\e(s)\ran ds\rt]\\
   &\q =\dbE\int_t^T|u^*(s)|^2ds + \dbE\int_t^T|u^*(s)|^2ds -2\,\dbE\int_t^T\lan u^*(s),u^*(s)\ran ds=0.
\end{align*}
This means that $\{u_\e(\cd)\}_{\e>0}$ converges strongly to $u^*(\cd)$ as $\e\to0$. $\hfill\qed$

\begin{remark}\rm
A similar result first appeared in \cite{Sun-Li-Yong 2016} for the classical stochastic LQ problem.
After that, Sun \cite{Sun 2017} extended it to the mean-field case.
More precisely, they found that if Problem (MF-SLQ) is open-loop solvable at $(t,\xi)$, then the limit of any weakly/strongly convergent subsequence of $\{u_\e(\cd)\}_{\e>0}$ is an open-loop optimal control for $(t,\xi)$.
Our result refines that in \cite{Sun 2017} by showing the family $\{u_\e(\cd)\}_{\e>0}$ itself is strongly convergent when Problem (MF-SLQ) is open-loop solvable. %This improvement serves as a crucial bridge to the weak closed-loop solvability presented in the next section.
\end{remark}

\section{Weak Closed-Loop Solvability}\label{Sec:4}
In this section, we shall establish the equivalence between open-loop and weak closed-loop solvabilities of Problem (MF-SLQ).
In fact, we will show that $\Th_\e(\cd)$, $\bar\Th_\e(\cd)$, $v_\e(\cd)$ defined by \rf{Th-e}, \rf{bar-Th-e} and \rf{v-e} converge locally in $(0,T)$, and that the limit $(\Th^*(\cd),\bar\Th^*(\cd),v^*(\cd))$ is a weak closed-loop optimal strategy.
Different from the classical stochastic LQ problems, there are two deterministic functions  $\Th_\e(\cd),\,\bar\Th_\e(\cd)$ and one $\dbF$-progressively measurable process $v_\e(\cd)$ in the optimal closed-loop strategy of Problem (MF-SLQ)$_\e$. In order to work separately with them, we introduce the following two lemmas.
The first one will enable us to work separately with $(\Th_\e(\cd),\bar\Th_\e(\cd))$ and $v_\e(\cd)$.
Recall that the associated Problem (MF-SLQ)$^0$ is to minimize \rf{cost0} subject to \rf{state0}.
\begin{lemma}\label{lmm:SLQ-SLQ0}
Let {\rm\ref{ass:A1}} and {\rm\ref{ass:A2}} hold.
If Problem {\rm(MF-SLQ)} is open-loop solvable, then so is Problem {\rm(MF-SLQ)$^0$}.
\end{lemma}

\begin{proof}
Let $(t,\xi)\in\cD$ be arbitrary initial pair.
By the definition of Problem (MF-SLQ)$^0$, we have $b(\cd),\bar b(\cd),\si(\cd),\bar\si(\cd),g,\bar g,q(\cd),\bar q(\cd),\rho(\cd).\bar\rho(\cd)=0$.
It follows that the solutions $(\eta_\e(\cd),\z_\e(\cd))$ to BSDE \rf{eta-zeta-e} and $\bar\eta_\e(\cd)$ to ODE \rf{bar-eta-e} are identically $(0,0)$ and $0$, respectively.
Hence the process $v_\e(\cd)$ defined by \rf{v-e} is identically zero.
So by \autoref{thm:SLQ-oloop-kehua}, to prove that Problem (MF-SLQ)$^0$ is open-loop solvable at $(t,\xi)$, we need to verify that the family $\{u_\e(\cd)\}_{\e>0}$ is bounded in $L^2_\dbF(t,T;\dbR^m)$ with $u_\e(\cd)=\Th_\e(\cd)X_\e(\cd)+\bar\Th_\e(\cd)\dbE[X_\e(\cd)]$, where $X_\e(\cd)$ is the solution to the following:
\bel{Xe-0} \left\{\begin{aligned}
         dX_\e(s) &=\big\{AX_\e(s)+\bar A\dbE[X_\e(s)]+B\big\{\Th_\e X_\e(s)+\bar\Th_\e\dbE[X_\e(s)]\big\} \\
                 &\hp{=\ }\qq +\bar B\dbE\big\{\Th_\e X_\e(s)+\bar\Th_\e\dbE[X_\e(s)]\big\}\big\}ds\\
                  &\hp{=\ }+\big\{CX_\e(s)+\bar C\dbE[X_\e(s)]+D\big\{\Th_\e X_\e(s)+\bar\Th_\e\dbE[X_\e(s)]\big\} \\
                 &\hp{=\ }\qq +\bar D\dbE\big\{\Th_\e X_\e(s)+\bar\Th_\e\dbE[X_\e(s)]\big\}\big\}dW(s),\\
          X_\e(t) &=\xi.
      \end{aligned}\right.\ee
To this end, we return to Problem (MF-SLQ).
Let $v_\e(\cd)$ be defined by \rf{v-e} and denote by $X_\e^{t,\xi}(\cd)$ and $X_\e^{t,0}(\cd)$ the solutions to \rf{closed-loop-syst-e}
with respect to the initial pairs $(t,\xi)$ and $(t,0)$, respectively.
Note that Problem (MF-SLQ) is open-loop solvable and hence is open-loop solvable at both $(t,\xi)$ and $(t,0)$. By \autoref{thm:SLQ-oloop-kehua},
the families
$$ u_\e^{t,\xi}(\cd)\deq \Th_\e(\cd)X_\e^{t,\xi}(\cd)+\bar\Th_\e(\cd)\dbE[X^{t,\xi}_\e(\cd)]+ v_\e(\cd) \q\hb{and}\q u_\e^{t,0}(\cd) \deq \Th_\e(\cd) X_\e^{t,0}(\cd) + \bar\Th_\e(\cd)\dbE[X^{t,0}_\e(\cd)]+v_\e(\cd)$$
are bounded in $L^2_\dbF(t,T;\dbR^m)$. Because the process $v_\e(\cd)$ is independent of the initial state, the difference $X_\e^{t,\xi}(\cd)-X_\e^{t,0}(\cd)$ satisfies the same SDE as $X_\e(\cd)$.
By the uniqueness of solutions of SDEs, we must have
$X_\e(\cd)= X_\e^{t,\xi}(\cd)-X_\e^{t,0}(\cd).$
It follows that
\begin{align*}
u_\e(\cd)&=\Th_\e(\cd)X_\e(\cd)+\bar\Th_\e(\cd)\dbE[X_\e(\cd)]\\
&=\Th_\e(\cd)[X_\e^{t,\xi}(\cd)-X_\e^{t,0}(\cd)]+\bar\Th_\e(\cd)\{\dbE[X^{t,\xi}_\e(\cd)]-\dbE[X^{t,0}_\e(\cd)]\}\\
&=u_\e^{t,\xi}(\cd)-u_\e^{t,0}(\cd).
\end{align*}
Because $\{u_\e^{t,\xi}(\cd)\}_{\e>0}$ and $\{u_\e^{t,0}(\cd)\}_{\e>0}$
are bounded in $L^2_\dbF(t,T;\dbR^m)$, so is $\{u_\e(\cd)\}_{\e>0}$. By \autoref{thm:SLQ-oloop-kehua} again, Problem (MF-SLQ)$^0$ is open-loop solvable. \end{proof}

The second one will help us to work separately with $\Th_\e(\cd)$ and $\bar\Th_\e(\cd)$.

\begin{lemma}\label{lemma-Cauchy}
For any $0<t<T$ and $\e>0$, let $F_{\e}(\cd):[t,T]\to\dbR^{m\times n}$ be a square-integrable deterministic function.
Suppose that $\{F_\e(\cd)\xi\}_{\e>0}$ is Cauchy in $L_{\dbF}^2(t,T;\dbR^m)$ for any bounded $\cF_t$-measurable random vector $\xi$ with $\dbE[\xi]=0$.
Then  $\{F(\cd)\}_{\e>0}$ is Cauchy in $L^2(t,T;\dbR^{m\times n})$.
\end{lemma}
\begin{proof}
Let $\Om_1$ be an $\cF_t$-measurable set satisfying $0<\dbP(\Om_1)<1$. Then the set $\Om_2\deq\Om\backslash\Om_1$ also
satisfies $0<\dbP(\Om_2)<1$. Let $e_1\deq (1,0,...,0)^\top\in\dbR^m$ and define
$$
\xi_1=e_1I_{\Om_1}-{\dbP(\Om_1)\over\dbP(\Om_2)}e_1I_{\Om_2}.
$$
It is clear to see that $\xi_1$ is an $\cF_t$-measurable random vector and
$$\dbE[\xi_1]=e_1\dbE[I_{\Om_1}]-{\dbP(\Om_1)\over\dbP(\Om_2)}e_1\dbE[I_{\Om_2}]=e_1\dbP(\Om_1)-e_1\dbP(\Om_1)=0.$$
For any $\e_1,\e_2>0$,
\begin{align*}
&\dbE\int_t^T|F_{\e_1}(s)\xi_1-F_{\e_2}(s)\xi_1|^2ds\\
&\q=\dbE\int_t^T\Big|[f^1_{\e_1}(s)-f^1_{\e_2}(s)]I_{\Om_1}-{\dbP(\Om_1)\over\dbP(\Om_2)}[f^1_{\e_1}(s)-f^1_{\e_2}(s)]I_{\Om_2}\Big|^2ds\\
&\q=\dbE\int_t^T\Big[\big|f^1_{\e_1}(s)-f^1_{\e_2}(s)\big|^2I_{\Om_1}+{\dbP(\Om_1)^2\over\dbP(\Om_2)^2}\big|f^1_{\e_1}(s)-f^1_{\e_2}(s)\big|^2I_{\Om_2}\Big]ds\\
&\q=\big(\dbP(\Om_1)+{\dbP(\Om_1)^2\over\dbP(\Om_2)}\big) \int_t^T\big|f^1_{\e_1}(s)-f^1_{\e_2}(s)\big|^2ds,
\end{align*}
where $(f_{\e}^1(\cd),...,f_{\e}^n(\cd))=F_{\e}(\cd);\e>0$.
Since $\{F_\e(\cd)\xi_1\}_{\e>0}$ is Cauchy in $L_{\dbF}^2(t,T;\dbR^m)$ and $\dbP(\Om_1)+{\dbP(\Om_1)^2\over\dbP(\Om_2)}>0$, the above implies that  $\{f_\e^1(\cd)\}_{\e>0}$ is Cauchy in $L^2(t,T;\dbR^{m})$.
Similarly, one can prove that $\{f_\e^i(\cd)\}_{\e>0}$ is Cauchy in $L^2(t,T;\dbR^{m})$ for $i=2,...,n$.
Hence $\{F(\cd)\}_{\e>0}$ is Cauchy in $L^2(t,T;\dbR^{m\times n})$.
\end{proof}
We now prove that the family $\{\ti\Th_\e(\cd)\}_{\e>0}$ defined by \rf{def-ti-Th-Th-v-e} is locally convergent in $[0,T)$.

\begin{proposition}\label{prop:limit-ti-The} Let {\rm\ref{ass:A1}} and {\rm\ref{ass:A2}} hold. Suppose that Problem {\rm(MF-SLQ)$^0$} is open-loop solvable.
Then the family $\{\ti\Th_\e(\cd)\}_{\e>0}$ defined by \rf{def-ti-Th-Th-v-e} converges in $L^2(0,T';\dbR^{m\times n})$ for any $0<T'<T$; that is, there exists a locally square-integrable deterministic function $\ti\Th^*(\cdot):[0,T)\to\dbR^{m\times n}$ such that
$$ \lim_{\e\to 0}\int_0^{T'}|\ti\Th_\e(s)-\ti\Th^*(s)|^2ds=0, \q\forall\, 0<T'<T. $$
\end{proposition}

\begin{proof}
We need to show that for any $0<T'<T$, the family $\{\ti\Th_\e(\cd)\}_{\e>0}$ is Cauchy in $L^2(0,T';\dbR^{m\times n})$.
For any $(t,\xi)\in\cD$, let $X_\e(\cd)$ be the unique solution to the closed-loop system
      \bel{closed-loop-syst-0-e}\left\{\begin{aligned}
         dX_\e(s) &=\big\{AX_\e(s)+\bar A\dbE[X_\e(s)]+B\big\{\Th_\e X_\e(s)+\bar\Th_\e\dbE[X_\e(s)]\big\} \\
                 &\hp{=\ }\qq +\bar B\dbE\big\{\Th_\e X_\e(s)+\bar\Th_\e\dbE[X_\e(s)]\big\}\big\}ds\\
                  &\hp{=\ }+\big\{CX_\e(s)+\bar C\dbE[X_\e(s)]+D\big\{\Th_\e X_\e(s)+\bar\Th_\e\dbE[X_\e(s)]\big\} \\
                 &\hp{=\ }\qq +\bar D\dbE\big\{\Th_\e X_\e(s)+\bar\Th_\e\dbE[X_\e(s)]\big\}\big\}dW(s),\\
          X_\e(t) &=\xi.
      \end{aligned}\right.\ee
In light of $\bar\Th_\e(\cd)\equiv\ti\Th_\e(\cd)-\Th_\e(\cd)$, by taking expectation on the both sides of the above, we have
\bel{expect-syst-0-e}\left\{\begin{aligned}
         d\dbE[X_\e(s)] &=(A+\bar A+B\ti\Th_\e+\bar B\ti\Th_\e)\dbE[X_\e(s)]ds,\q s\in[t,T], \\
          \dbE[X_\e(t)] &=\dbE[\xi].
\end{aligned}\right.\ee
Let $\F_\e(\cd)\in C([t,T];\dbR^{n\times n})$ be the solution to the following ODE:
\bel{Phi-e}\left\{\begin{aligned}
   d\Phi_\e(s) &=(A+\bar A+B\ti\Th_\e+\bar B\ti\Th_\e)\Phi_\e(s)ds, \q s\in[t,T], \\
    \Phi_\e(t) &= I_n.
\end{aligned}\right.\ee
%%
%\bel{Phi-e}\left\{\begin{aligned}
%   d\Phi_\e(s) &= [A(s)+B(s)\Th_\e(s)]\Phi_\e(s)ds \\
%                &\hp{=\ } +[C(s)+D(s)\Th_\e(s)]\Phi_\e(s)dW(s), \q s\in[t,T],\\
%                  %
%    \Phi_\e(t) &= I_n.
%\end{aligned}\right.\ee
%
Then for any initial state $\xi$, the expectation of the solution to \rf{closed-loop-syst-0-e} (i.e., the solution to \rf{expect-syst-0-e})  is given by
$$ \dbE[X_\e(s)] = \F_\e(s)\dbE[\xi], \q s\in[t,T]. $$
Since Problem (MF-SLQ)$^0$ is open-loop solvable, by \autoref{thm:SLQ-oloop-kehua}, the family
\begin{align*}u_\e(s) &=\Th_\e(s)X_\e(s)+\bar\Th_\e(s)\dbE[X_\e(s)]\\
&=\Th_\e(s)\{X_\e(s)-\dbE[X_\e(s)]\}+\ti\Th_\e(s)\dbE[X_\e(s)],\q s\in[t,T]; \q \e>0
\end{align*}
is strongly convergent in $L^2_\dbF(t,T;\dbR^m)$ for any $\xi\in L^2_{\cF_t}(\Om;\dbR^n)$.
Hence the family of functions
$$\dbE[u_\e(s)] =\ti\Th_\e(s)\dbE[X_\e(s)]=\ti\Th_\e(s)\F_\e(s)\dbE[\xi],\q s\in[t,T]; \q \e>0$$
is strongly convergent in $L^2(t,T;\dbR^m)$ for any $\xi\in L^2_{\cF_t}(\Om;\dbR^n)$.
It follows that $\{\ti\Th_\e(\cd)\F_\e(\cd)\}_{\e>0}$
converges strongly in $L^2(t,T;\dbR^{m\times n})$ as $\e\to0$. Denote $\ti U_\e(\cd)=\ti\Th_\e(\cd)\F_\e(\cd)$ and let $\ti U^*(\cd)$ be the strong limit of $\ti U_\e(\cd)$.
By the stability results of ODE, the family of continuous functions $\F_\e(\cd)$  converges uniformly to the solution  of
$$\left\{\begin{aligned}
    d\Phi^*(s)  &= \big\{(A(s)+\bar A(s))\Phi^*(s) + (B(s)+\bar B(s))\ti U^*(s)\big\}ds, \q s\in[t,T], \\
       \Phi^*(t) &= I_n.
\end{aligned}\right.$$
By noting that $\Phi^*(t)=I_n$, there is a small constant $\D_t>0$ such that for any small $\e>0$,
\begin{enumerate}[(a)]
\item $\F_\e(s)$ is invertible for all $s\in[t,t+\D_t]$,  and
\item $|\F_\e(s)|\ges {1\over2}$ for all $s\in[t,t+\D_t]$.
\end{enumerate}
We claim that the family $\{\ti\Th_\e(\cd)\}_{\e>0}$ is Cauchy in $L^2(t,t+\D_t;\dbR^{m\times n})$. Indeed, by (a) and (b), we have
\begin{align*}
& \int_t^{t+\D_t} |\ti\Th_{\e_1}(s)-\ti\Th_{\e_2}(s)|^2 ds \\
&\q= \int_t^{t+\D_t} \lt|\ti U_{\e_1}(s)\Phi_{\e_1}(s)^{-1}-\ti U_{\e_2}(s)\Phi_{\e_2}(s)^{-1}\rt|^2ds \\
&\q\les 2\int_t^{t+\D_t} \big|\ti U_{\e_1}(s)-\ti U_{\e_2}(s)\big|^2  \big|\Phi_{\e_1}(s)^{-1}\big|^2 ds  +2\int_t^{t+\D_t} \big|\ti U_{\e_2}(s)\big|^2 \big|\Phi_{\e_1}(s)^{-1}-\Phi_{\e_2}(s)^{-1}\big|^2ds\\
&\q\les 2\int_t^{t+\D_t} \big|\ti U_{\e_1}(s)-\ti U_{\e_2}(s)\big|^2  \big|\Phi_{\e_1}(s)^{-1}\big|^2ds \\
&\q\hp{\les\ } +2\int_t^{t+\D_t} \big|\ti U_{\e_2}(s)\big|^2 \big|\Phi_{\e_1}(s)^{-1}\big|^2 \big|\Phi_{\e_2}(s)-\Phi_{\e_1}(s)\big|^2 \big|\Phi_{\e_2}(s)^{-1}\big|^2ds  \\
&\q\les 8\int_t^{t+\D_t} \big|\ti U_{\e_1}(s)-\ti U_{\e_2}(s)\big|^2 ds+32\int_t^{t+\D_t} \big|\ti U_{\e_2}(s)\big|^2 ds  \cd \sup_{t\les s\les t+\D_t}\big|\Phi_{\e_1}(s)-\Phi_{\e_2}(s)\big|^2.
\end{align*}
Since $\{\ti U_\e(\cd)\}_{\e>0}$ is Cauchy in $L^2_\dbF(t,T;\dbR^{m\times n})$ and $\{\F_\e(\cd)\}_{\e>0}$ converges uniformly on $[t,T]$,
the last two terms of the above inequality approach to zero as $\e_1,\e_2\to0$.

\ms

We next use a compactness argument to prove that $\{\ti\Th_\e(\cd)\}_{\e>0}$ is Cauchy in $L^2(0,T';\dbR^{m\times n})$
for any $0<T'<T$. From the preceding argument we see that for each $t\in[0,T']$, there exists a small $\D_t>0$ such that $\{\ti\Th_\e(\cd)\}_{\e>0}$ is Cauchy in $L^2(t,t+\D_t;\dbR^{m\times n})$. Since $[0,T']$ is compact, we can choose finitely many $t\in[0,T']$ and $\D_t$, say, $t_1,t_2,\ldots,t_k;\D_{t_1},\D_{t_2},\ldots,\D_{t_k}$, such that
$\{\ti\Th_\e(\cd)\}_{\e>0}$ is Cauchy in each $L^2(t_j,t_j+\D_{t_j};\dbR^{m\times n})$ and $[0,T']\subseteq\bigcup_{j=1}^k[t_j,t_j+\D_{t_j}]$. It follows that
$$ \int_0^{T^\prime} |\ti\Th_{\e_1}(s)-\ti\Th_{\e_2}(s)|^2 ds
\les \sum_{j=1}^k\int_{t_j}^{t_j+\D_{t_j}} |\ti\Th_{\e_1}(s)-\ti\Th_{\e_2}(s)|^2 ds \to0 \q\hb{as}\q \e_1,\e_2\to0. $$
The proof is therefore completed.
\end{proof}
The following result shows that $\{\Th_\e(\cd)\}_{\e>0}$ defined by \rf{Th-e} is locally convergent in $(0,T)$.

\begin{proposition}\label{prop:limit-The} Let {\rm\ref{ass:A1}} and {\rm\ref{ass:A2}} hold. Suppose that Problem {\rm(MF-SLQ)$^0$} is open-loop solvable.
Then the family $\{\Th_\e(\cd)\}_{\e>0}$ defined by \rf{Th-e} converges in $L^2(t,T^\prime;\dbR^{m\times n})$ for any $0<t<T'<T$; that is, there exists a locally square-integrable deterministic function $\Th^*(\cdot):(0,T)\to\dbR^{m\times n}$ such that
$$ \lim_{\e\to 0}\int_{t}^{T'}|\Th_\e(s)-\Th^*(s)|^2ds=0, \q\forall\, 0<t<T'<T. $$
\end{proposition}
\begin{proof}
To do that, we just need to show that for any $0<t<T'<T$, the family $\{\Th_\e(\cd)\}_{\e>0}$ is Cauchy in $L^2(t,T';\dbR^{m\times n})$.
For any $0<t<T$ and $\xi\in L^2_{\cF_t}(\Om;\dbR^n)$ with $\dbE[\xi]=0$, let $X_\e(\cd)$ be the unique solution to the closed-loop system
      \bel{closed-loop-syst-0-e-Th}\left\{\begin{aligned}
         dX_\e(s) &=\big\{AX_\e(s)+\bar A\dbE[X_\e(s)]+B\big\{\Th_\e X_\e(s)+\bar\Th_\e\dbE[X_\e(s)]\big\} \\
                 &\hp{=\ }\qq +\bar B\dbE\big\{\Th_\e X_\e(s)+\bar\Th_\e\dbE[X_\e(s)]\big\}\big\}ds\\
                  &\hp{=\ }+\big\{CX_\e(s)+\bar C\dbE[X_\e(s)]+D\big\{\Th_\e X_\e(s)+\bar\Th_\e\dbE[X_\e(s)]\big\} \\
                 &\hp{=\ }\qq +\bar D\dbE\big\{\Th_\e X_\e(s)+\bar\Th_\e\dbE[X_\e(s)]\big\}\big\}dW(s),\\
          X_\e(t) &=\xi.
      \end{aligned}\right.\ee
Taking expectation on the both sides of the above, we have
    \bel{closed-loop-syst-0-e-Th-expect}\left\{\begin{aligned}
         d\dbE[X_\e(s)] &=\big\{A+\bar A+B(\Th_\e +\bar\Th_\e)+\bar B(\Th_\e +\bar\Th_\e)\big\}\dbE[X_\e(s)]ds,\q s\in[t,T], \\
          \dbE[X_\e(t)] &=0,
      \end{aligned}\right.\ee
which implies that $\dbE[X_\e(s)]\equiv0;t\les s\les T$. Then the state equation \rf{closed-loop-syst-0-e-Th} can be rewritten as
      \bel{rw-syst-0-e-Th}\left\{\begin{aligned}
         dX_\e(s) &=\big\{AX_\e(s)+B\Th_\e X_\e(s)\big\}ds+\big\{CX_\e(s)+D\Th_\e X_\e(s)\big\}dW(s),\\
          X_\e(t) &=\xi.
      \end{aligned}\right.\ee
Let $\F_\e(\cd)\in L_\dbF^2(\Om;C([t,T];\dbR^{n\times n})$ be the solution to the following SDE:
\bel{Phi-e-Th}\left\{\begin{aligned}
   d\Phi_\e(s) &= (A+B\Th_\e)\Phi_\e(s)ds +(C+D\Th_\e)\Phi_\e(s)dW(s), \q s\in[t,T],\\
    \Phi_\e(t) &= I_n.
\end{aligned}\right.\ee
%%
%\bel{Phi-e}\left\{\begin{aligned}
%   d\Phi_\e(s) &= [A(s)+B(s)\Th_\e(s)]\Phi_\e(s)ds \\
%                &\hp{=\ } +[C(s)+D(s)\Th_\e(s)]\Phi_\e(s)dW(s), \q s\in[t,T],\\
%                  %
%    \Phi_\e(t) &= I_n.
%\end{aligned}\right.\ee
%
Clearly, for any initial state $\xi$, the solution of \rf{rw-syst-0-e-Th} (or \rf{closed-loop-syst-0-e-Th}) can be expressed by
$$ X_\e(s) = \F_\e(s)\xi, \q s\in[t,T]. $$
Since Problem (MF-SLQ)$^0$ is open-loop solvable, by \autoref{thm:SLQ-oloop-kehua}, the family
$$u_\e(s) =\Th_\e(s)X_\e(s) = \Th_\e(s)\F_\e(s)\xi,\q s\in[t,T]; \q \e>0$$
is strongly convergent in $L^2_\dbF(t,T;\dbR^m)$ for any $\xi\in L^2_{\cF_t}(\Om;\dbR^n)$ with $\dbE[\xi]=0$.
Denote $U_\e(\cd)=\Th_\e(\cd)\dbE[\F_\e(\cd)]$.  Note that $\F_\e(\cd)$ is independent of $\cF_t$, then for any $\e_1,\e_2>0$,
\begin{align}
&\nn\dbE\int_t^T\big|\Th_{\e_1}(s)\dbE[\Phi_{\e_1}(s)]\xi-\Th_{\e_2}(s)\dbE[\Phi_{\e_2}(s)]\xi\big|^2ds\nn\\
\nn&\q=\dbE\int_t^T\big|\Th_{\e_1}(s)\dbE_t[\Phi_{\e_1}(s)]\xi-\Th_{\e_2}(s)\dbE_t[\Phi_{\e_2}(s)]\xi\big|^2ds\\
\nn&\q=\dbE\int_t^T\big|\dbE_t[\Th_{\e_1}(s)\Phi_{\e_1}(s)\xi-\Th_{\e_2}(s)\Phi_{\e_2}(s)\xi]\big|^2ds\\
&\q\les\dbE\int_t^T\dbE_t\big[\big|\Th_{\e_1}(s)\Phi_{\e_1}(s)\xi-\Th_{\e_2}(s)\Phi_{\e_2}(s)\xi\big|^2\big]ds\nn\\
&\q=\dbE\int_t^T\big|\Th_{\e_1}(s)\Phi_{\e_1}(s)\xi-\Th_{\e_2}(s)\Phi_{\e_2}(s)\xi\big|^2ds.
\end{align}
It follows that $\{U_\e(\cd)\xi\}_{\e>0}=\{\Th_\e(\cd)\dbE[\Phi_\e(\cd)]\xi\}_{\e>0}$ is strongly convergent in $L^2_\dbF(t,T;\dbR^m)$ for any $\xi\in L^2_{\cF_t}(\Om;\dbR^n)$ with $\dbE[\xi]=0$.
By \autoref{lemma-Cauchy}, $\{U_\e(\cd)\}_{\e>0}\equiv\{\Th_\e(\cd)\dbE[\Phi_\e(\cd)]\}_{\e>0}$ converges strongly in $L^2(t,T;\dbR^{m\times n})$ as $\e\to0$.
Denote the limit of $U_\e(\cd)$ by $U^*(\cd)$. One sees that $\dbE[\F_\e(\cd)]$ satisfies the following ODE:
$$\left\{\begin{aligned}
    d\dbE[\Phi_\e(s)]  &= \{A(s)\dbE[\Phi_\e(s)] + B(s)U_\e(s)\}ds, \q s\in[t,T], \\
       \dbE[\Phi_\e(t)] &= I_n.
\end{aligned}\right.$$
Then the family of continuous functions $\dbE[\F_\e(\cd)]$  converges uniformly to the solution  of
$$\left\{\begin{aligned}
    d\dbE[\Phi^*(s)]  &= \{A(s)\dbE[\Phi^*(s)] + B(s)U^*(s)\}ds, \q s\in[t,T], \\
       \dbE[\Phi^*(t)] &= I_n.
\end{aligned}\right.$$
%Then very similar to the proof of \autoref{prop:limit-ti-The}, one can obtain that $\{\Th_\e(\cd)\}_{\e>0}$ is Cauchy in $L^2(t,T';\dbR^{m\times n})$.
%
Thus by noting that $\dbE[\Phi^*(t)]=I_n$ we can choose a small constant $\D_t>0$ such that for small $\e>0$,
\begin{enumerate}[(a)]
\item $\dbE[\F_\e(s)]$ is invertible for all $s\in[t,t+\D_t]$,  and
\item $|\dbE[\F_\e(s)]|\ges {1\over2}$ for all $s\in[t,t+\D_t]$.
\end{enumerate}
Since $\{U_\e(\cd)\}_{\e>0}$ is Cauchy in $L^2_\dbF(t,T;\dbR^{m\times n})$ and $\{\dbE[\F_\e(\cd)]\}_{\e>0}$ converges uniformly on $[t,T]$, combining (a) and (b), we have
\begin{align*}
& \int_t^{t+\D_t} |\Th_{\e_1}(s)-\Th_{\e_2}(s)|^2 ds \\
&\q= \int_t^{t+\D_t} \lt|U_{\e_1}(s)\dbE[\Phi_{\e_1}(s)]^{-1}-U_{\e_2}(s)\dbE[\Phi_{\e_2}(s)]^{-1}\rt|^2ds \\
 %              %
&\q\les 2\int_t^{t+\D_t} \big|U_{\e_1}(s)-U_{\e_2}(s)\big|^2  \big|\dbE[\Phi_{\e_1}(s)]^{-1}\big|^2 ds  +2\int_t^{t+\D_t} \big|U_{\e_2}(s)\big|^2 \big|\dbE[\Phi_{\e_1}(s)]^{-1}-\dbE[\Phi_{\e_2}(s)]^{-1}\big|^2ds\\
&\q\les 2\int_t^{t+\D_t} \big|U_{\e_1}(s)-U_{\e_2}(s)\big|^2  \big|\dbE[\Phi_{\e_1}(s)]^{-1}\big|^2ds \\
&\q\hp{\les\ } +2\int_t^{t+\D_t} \big|U_{\e_2}(s)\big|^2 \big|\dbE[\Phi_{\e_1}(s)^{-1}]\big|^2 \big|\dbE[\Phi_{\e_2}(s)]-\dbE[\Phi_{\e_1}(s)]\big|^2 \big|\dbE[\Phi_{\e_2}(s)^{-1}]\big|^2ds  \\
&\q\les 8\int_t^{t+\D_t} \big|U_{\e_1}(s)-U_{\e_2}(s)\big|^2 ds+32\int_t^{t+\D_t} \big|U_{\e_2}(s)\big|^2 ds  \cd \sup_{t\les s\les t+\D_t}\big|\dbE[\Phi_{\e_1}(s)]-\dbE[\Phi_{\e_2}(s)]\big|^2\\
&\q\to 0,\q\hbox{as}\q\e_1,\e_2\to0.
\end{align*}
This means that $\{\Th_\e(\cd)\}_{\e>0}$ is Cauchy in $L^2(t,t+\D_t;\dbR^{m\times n})$.
Similar to the last paragraph in the proof of \autoref{prop:limit-ti-The}, one can obtain that $\{\Th_\e(\cd)\}_{\e>0}$ is Cauchy in $L^2(t,T';\dbR^{m\times n})$
by the compactness argument.
\end{proof}

Combining \autoref{prop:limit-ti-The} with \autoref{prop:limit-The}, we have the following corollary, which shows that $\{\bar\Th_\e(\cd)\}_{\e>0}$ defined by \rf{bar-Th-e} is locally convergent in $(0,T)$.
\begin{corollary}\label{coroll:limit-bar-The}
Let {\rm\ref{ass:A1}} and {\rm\ref{ass:A2}} hold. Suppose that Problem {\rm(MF-SLQ)$^0$} is open-loop solvable.
Then the family $\{\bar\Th_\e(\cd)\}_{\e>0}$ defined by \rf{bar-Th-e} converges in $L^2(t,T^\prime;\dbR^{m\times n})$ for any $0<t<T'<T$; that is, there exists a locally square-integrable deterministic function $\bar\Th^*(\cdot):(0,T)\to\dbR^{m\times n}$ such that
$$ \lim_{\e\to 0}\int_{t}^{T'}|\bar\Th_\e(s)-\bar\Th^*(s)|^2ds=0, \q\forall\, 0<t<T'<T. $$
\end{corollary}

The next result shows that the family $\{v_\e(\cd)\}_{\e>0}$ defined by \rf{v-e} is also locally convergent in $(0,T)$.
\begin{proposition}\label{prop:limit-ve}
Let {\rm\ref{ass:A1}} and {\rm\ref{ass:A2}} hold. Suppose that Problem {\rm(MF-SLQ)} is open-loop solvable. Then the family $\{v_\e(\cd)\}_{\e>0}$ defined by \rf{v-e} converges in $L_\dbF^2(t,T^\prime;\dbR^m)$
for any $0<t<T^\prime<T$; that is, there exists a locally square-integrable process
$v^*(\cdot):(0,T)\times\Om\to\dbR^m$ such that
$$ \lim_{\e\to0}\dbE\int_t^{T^\prime}|v_\e(s)-v^*(s)|^2ds=0, \q\forall\, 0<t<T^\prime<T. $$
\end{proposition}

\begin{proof}
Let $X_\e(s)$; $0<t\les s\les T$ be the solution to the closed-loop system \rf{closed-loop-syst-e} with respect to any fixed initial pair $(t,\xi)$.
Since Problem (MF-SLQ) is open-loop solvable, by \autoref{thm:SLQ-oloop-kehua}, the family
$$ u_\e(s) = \Th_\e(s)X_\e(s) +\bar\Th_\e(s)\dbE[X_{\e}(s)]+ v_\e(s), \q s\in[t,T]; \q\e>0 $$
is Cauchy in $L^2_\dbF(t,T;\dbR^m)$. In other words,
$$\dbE\int_t^T|u_{\e_1}(s)-u_{\e_2}(s)|^2ds \to 0  \q\hb{as}\q  \e_1,\e_2 \to0. $$
By \autoref{lmm:well-posedness-SDE}, the above implies that
\bel{X1-X2=go0} \dbE\lt[\sup_{t\les s\les T}|X_{\e_1}(s)-X_{\e_2}(s)|^2\rt]
\les K\dbE\int^T_t|u_{\e_1}(s)-u_{\e_2}(s)|^2ds \to0  \q\hb{as}\q  \e_1,\e_2 \to0. \ee
Now take any $0<t<T^\prime<T$.
Since Problem (MF-SLQ) is open-loop solvable, by \autoref{lmm:SLQ-SLQ0},
Problem (MF-SLQ)$^0$  is open-loop solvable.
Then according to \autoref{prop:limit-The} and \autoref{coroll:limit-bar-The}, the families $\{\Th_\e(\cd)\}_{\e>0}$ and $\{\bar\Th_\e(\cd)\}_{\e>0}$ are both Cauchy in $L^2(t,T^\prime;\dbR^{m\times n})$.
Thus, making use of \rf{X1-X2=go0}, we obtain
\begin{align*}
& \dbE\int_t^{T^\prime}|\Th_{\e_1}(s)X_{\e_1}(s)-\Th_{\e_2}(s)X_{\e_2}(s)|^2ds \\
&\q\les 2\dbE\int_t^{T^\prime}|\Th_{\e_1}(s)-\Th_{\e_2}(s)|^2|X_{\e_1}(s)|^2ds
        + 2\dbE\int_t^{T^\prime}|\Th_{\e_2}(s)|^2|X_{\e_1}(s)-X_{\e_2}(s)|^2 ds \\
&\q\les 2\int_t^{T^\prime}|\Th_{\e_1}(s)-\Th_{\e_2}(s)|^2ds\cd \dbE\left[ \sup_{t\les s\les T^\prime}|X_{\e_1}(s)|^2\right] \\
&\q\hp{\les\ } + 2\int_t^{T^\prime}|\Th_{\e_2}(s)|^2ds\cd\dbE\left[\sup_{t\les s\les T^\prime}|X_{\e_1}(s)-X_{\e_2}(s)|^2\right] \\
&\q\to 0  \q\hb{as}\q  \e_1,\e_2 \to0,
\end{align*}
and
\begin{align*}
& \dbE\int_t^{T^\prime}\big|\bar\Th_{\e_1}(s)\dbE[X_{\e_1}(s)]-\bar\Th_{\e_2}(s)\dbE[X_{\e_2}(s)]\big|^2ds \\
&\q\les 2\int_t^{T^\prime}|\bar\Th_{\e_1}(s)-\bar\Th_{\e_2}(s)|^2\dbE[|X_{\e_1}(s)|^2]ds
        + 2\int_t^{T^\prime}|\bar\Th_{\e_2}(s)|^2\dbE[|X_{\e_1}(s)-X_{\e_2}(s)|^2] ds \\
&\q\les 2\int_t^{T^\prime}|\bar\Th_{\e_1}(s)-\bar\Th_{\e_2}(s)|^2ds\cd \sup_{t\les s\les T^\prime}\dbE\left[ |X_{\e_1}(s)|^2\right] \\
&\q\hp{\les\ } + 2\int_t^{T^\prime}|\bar\Th_{\e_2}(s)|^2ds\cd\sup_{t\les s\les T^\prime}\dbE\left[|X_{\e_1}(s)-X_{\e_2}(s)|^2\right] \\
&\q\to 0  \q\hb{as}\q  \e_1,\e_2 \to0.
\end{align*}
Hence,
\begin{align*}
& \dbE\int_t^{T^\prime}|v_{\e_1}(s)-v_{\e_2}(s)|^2ds \\
&\q=\dbE\int_t^{T^\prime}\big|\{u_{\e_1}(s)-\Th_{\e_1}(s)X_{\e_1}(s)-\bar\Th_{\e_1}(s)\dbE[X_{\e_1}(s)]\}\\
&\qq\qq\qq-\{u_{\e_2}(s)-\Th_{\e_2}(s)X_{\e_2}(s)-\bar\Th_{\e_2}(s)\dbE[X_{\e_1}(s)]\}\big|^2ds \\
&\q\les 3\dbE\int_t^{T^\prime}|u_{\e_1}(s)-u_{\e_2}(s)|^2 + 3\dbE\int_t^{T^\prime}|\Th_{\e_1}(s)X_{\e_1}(s)-\Th_{\e_2}(s)X_{\e_2}(s)|^2ds \\
&\qq +3\dbE\int_t^{T^\prime}\big|\bar\Th_{\e_1}(s)\dbE[X_{\e_1}(s)]-\bar\Th_{\e_2}(s)\dbE[X_{\e_2}(s)]\big|^2ds \\
&\q\to 0  \q\hb{as}\q  \e_1,\e_2 \to0.
\end{align*}
This shows that the family $\{v_\e(\cd)\}_{\e>0}$ converges in $L_\dbF^2(t,T^\prime;\dbR^m)$.
\end{proof}

Now we present the main result of this section, which establishes the equivalence between open-loop and weak closed-loop solvabilities of Problem (MF-SLQ).

\begin{theorem}\label{thm:open=weak-closed}
Let {\rm\ref{ass:A1} and \ref{ass:A2}} hold. If Problem {\rm(MF-SLQ)} is open-loop solvable, then the limit triple $(\Th^*(\cd),\bar\Th^*(\cd),v^*(\cd))$ obtained in Propositions \ref{prop:limit-The}, \ref{prop:limit-ve} and Corollary \ref{coroll:limit-bar-The} is a weak closed-loop optimal strategy of Problem {\rm(MF-SLQ)} on any $(t,T)$. Consequently, the open-loop and weak closed-loop solvabilities of Problem {\rm(MF-SLQ)} are equivalent.
\end{theorem}

\begin{proof}
For any initial pair $(t,\xi)\in\cD$,  let $\{u_\e(s);t\les s\les T\}_{\e>0}$
be the family defined by \rf{u-e}.  Since Problem (MF-SLQ) is open-loop solvable at $(t,\xi)$,
\autoref{thm:SLQ-oloop-kehua} implies that $\{u_\e(s);t\les s\les T\}_{\e>0}$ converges strongly to an open-loop
optimal control $\{u^*(s);t\les s\les T\}$ of Problem (MF-SLQ) (for the initial pair $(t,\xi)$).
Let $\{X^*(s);t\les s\les T\}$ be the solution to
$$\left\{\begin{aligned}
   dX^*(s) &=\big\{A(s)X^*(s)+\bar A(s)\dbE[X^*(s)]+ B(s)u^*(s)+\bar B(s)\dbE[u^*(s)]+ b(s)\big\}ds\\
         &\hp{=\ } +\big\{C(s)X^*(s)+\bar C(s)\dbE[X^*(s)]+ D(s)u^*(s)+\bar D(s)\dbE[u^*(s)]+ \si(s)\big\}dW(s),\q s\in[t,T],\\
     X^*(t)&= \xi,
\end{aligned}\right.$$
then $X^*(\cd)$ is the optimal state process.
If we can show that
\bel{18-6-1}u^*(s)=\Th^*(s)X^*(s)+\bar\Th^*(s)\dbE[X^*(s)]+v^*(s), \q t\les s<T,\ee
then $(\Th^*(\cd),\bar\Th^*(\cd),v^*(\cd))$ is clearly a weak closed-loop optimal strategy of Problem {\rm(MF-SLQ)} on $(t,T)$.
To do this, we note that by \autoref{lmm:well-posedness-SDE},
$$\dbE\lt[\sup_{t\les s\les T}|X_\e(s)-X^*(s)|^2\rt] \les K\dbE\int^T_t|u_\e(s)-u^*(s)|^2ds \to0  \q\hb{as}\q  \e\to0,$$
where $\{X_\e(s);t\les s\les T\}$ is the solution to equation \rf{closed-loop-syst-e}.
Further, by  \autoref{prop:limit-The}, \autoref{coroll:limit-bar-The} and \autoref{prop:limit-ve},
\begin{align*}
&\lim_{\e\to 0}\int_{t^\prime}^{T^\prime}|\Th_\e(s)-\Th^*(s)|^2ds=0, \q\forall\, 0<t^\prime<T^\prime<T, \\
&\lim_{\e\to 0}\int_{t^\prime}^{T^\prime}|\bar\Th_\e(s)-\bar\Th^*(s)|^2ds=0, \q\forall\, 0<t^\prime<T^\prime<T, \\
&\lim_{\e\to 0}\dbE\int_{t^\prime}^{T^\prime}|v_\e(s)-v^*(s)|^2ds=0, \q\forall\, 0<t^\prime<T^\prime<T.
\end{align*}
It follows that for any $0\les t<t^\prime<T^\prime<T$,
\begin{align*}
&\dbE\int_{t^\prime}^{T^\prime} \big|[\Th_\e(s)X_\e(s)+\bar\Th_\e(s)\dbE[X_\e(s)]+v_\e(s)]-[\Th^*(s)X^*(s)+\bar\Th^*(s)\dbE[X^*(s)]+v^*(s)]\big|^2ds \\
&\q\les 3\dbE\int_{t^\prime}^{T^\prime} |\Th_\e(s)X_\e(s)-\Th^*(s)X^*(s)|^2ds+3\dbE\int_{t^\prime}^{T^\prime} \big|\bar\Th_\e(s)\dbE[X_\e(s)]-\bar\Th^*(s)\dbE[X^*(s)]\big|^2ds  \\
&\qq\q+3\dbE\int_{t^\prime}^{T^\prime} |v_\e(s)-v^*(s)|^2ds \\
&\q\les 3\dbE\int_{t^\prime}^{T^\prime} |v_\e(s)-v^*(s)|^2ds +6\dbE\int_{t^\prime}^{T^\prime} |\Th_\e(s)|^2|X_\e(s)-X^*(s)|^2ds \\
&\q\hp{\les\ } + 6\dbE\int_{t^\prime}^{T^\prime} |\Th_\e(s)-\Th^*(s)|^2|X^*(s)|^2ds+6\dbE\int_{t^\prime}^{T^\prime} |\bar\Th_\e(s)|^2\big|\dbE[X_\e(s)]-\dbE[X^*(s)]\big|^2ds \\
&\q\hp{\les\ }+ 6\int_{t^\prime}^{T^\prime} |\bar\Th_\e(s)-\bar\Th^*(s)|^2|\dbE[X^*(s)]|^2ds\\
&\q\les 3\dbE\int_{t^\prime}^{T'} |v_\e(s)-v^*(s)|^2ds
        +6\int_{t^\prime}^{T'}|\Th_\e(s)|^2ds \cd\dbE\[\sup_{t\les s\les T}|X_\e(s)-X^*(s)|^2\] \\
&\q\hp{\les\ } +6\int_{t^\prime}^{T^\prime} |\Th_\e(s)-\Th^*(s)|^2ds \cd\dbE\[\sup_{t\les s\les T}|X^*(s)|^2\]\\
&\q\hp{\les\ }+6\int_{t^\prime}^{T'}|\bar\Th_\e(s)|^2ds \cd\[\sup_{t\les s\les T}|\dbE[X_\e(s)]-\dbE[X^*(s)|^2\] \\
&\q\hp{\les\ } +6\int_{t^\prime}^{T^\prime} |\bar\Th_\e(s)-\bar\Th^*(s)|^2ds \cd\sup_{t\les s\les T}\dbE\[|X^*(s)|^2\]\\
&\q\to 0 \q\hb{as}\q \e\to0.
\end{align*}
Recall that $u_\e(s)=\Th_\e(s)X_\e(s)+\bar\Th_\e(s)\dbE[X_\e(s)]+v_\e(s);t\les s\les T$ converges strongly to $u^*(s);t\les s\les T$ in $L^2_\dbF(t,T;\dbR^m)$ as $\e\to0$. Then \rf{18-6-1} must hold and hence $(\Th^*(\cd),\bar\Th^*(\cd),v^*(\cd))$ is a weak closed-loop optimal strategy.
The above argument shows that the open-loop solvability implies the weak closed-loop solvability.
The reverse implication is obvious by \autoref{def-wcloop}.
\end{proof}

\section{An Example}\label{Sec:Example}
In this section, we present an example to  illustrate the result
we obtained. In the example, the LQ problem is open-loop solvable
(and hence weakly closed-loop solvable) but not closed-loop solvable.
Using the method introduced in \autoref{thm:open=weak-closed}, we  find a weak closed-loop optimal strategy.
\begin{example}\label{ex-5.1}\rm
Consider the following Problem (MF-SLQ) with one-dimensional state equation
$$\left\{\begin{aligned}
   dX(s) &=\big\{-X(s)+\dbE[X(s)]+u(s)+\dbE[u(s)]\big\}ds\\
   & \q+\big\{ \sqrt{2}X(s)-\sqrt{2}\dbE[X(s)]\big\}dW(s), \q s\in[t,1],\\
    X(t) &= \xi,
\end{aligned}\right.$$
and cost functional
$$ J(t,\xi;u(\cd))=\dbE |X(1)|^2+|\dbE[X(s)]|^2. $$

We first claim that this LQ problem is not closed-loop solvable on any $[t,1]$.
Indeed, the generalized Riccati equation associated with this problem reads
$$\left\{\begin{aligned}
   \dot P(s) &=P(s)0^\dag P(s)=0, \q s\in[t,1],\\
   \dot \Pi(s) &=\Pi(s)0^\dag \Pi(s)=0, \q s\in[t,1],\\
        P(1) &=1,\q \Pi(1)=2,
\end{aligned}\right.$$
whose solution is, obviously, $(P(s),\Pi(s))\equiv (1,2)$. For any $s\in[t,1]$, we have
\begin{align*}
   &\sR\big(B(s)^\top P(s)+D(s)^\top P(s)C(s)+S(s)\big)=\sR(1)=\dbR, \\
   &\sR\big(R(s)+D(s)^\top P(s)D(s)\big)=\sR(0)=\{0\};\\
    &\sR\big((B(s)+\bar B(s))^\top \Pi(s)+(D(s)+\bar D(s))^\top P(s)(C(s)+\bar C(s))+S(s)+\bar S(s)\big)=\sR(4)=\dbR, \\
   &\sR\big(R(s)+\bar R(s)+(D(s)+\bar D(s))^\top P(s)(D(s)+\bar D(s))\big)=\sR(0)=\{0\},
\end{align*}
where $\sR(M)$ denotes the range of a matrix $M$.
By \cite[Definition 2.8]{Li-Sun-Yong 2016}, $(P(\cd),\Pi(\cd))$ is not a regular solution.
Our claim then follows from \cite[Theorem 4.1]{Li-Sun-Yong 2016}.

\ms

Next we use \autoref{thm:SLQ-oloop-kehua} to conclude that the above LQ problem is open-loop solvable
and hence, by \autoref{thm:open=weak-closed}, weakly closed-loop solvable.
Without loss of generality, we consider only the open-loop solvability at $t=0$.
To this end, let $\e>0$ be arbitrary and consider the Riccati equation \rf{Ric-e}, which, in our example, read:
\bel{ex-Pe}\left\{\begin{aligned}
   \dot P_\e(s) &={P_\e(s)^2\over\e}, \q s\in[0,1],\\
    \dot \Pi_\e(s) &={4\Pi_\e(s)^2\over\e}, \q s\in[0,1],\\
        P_\e(1) &=1,\q \Pi_\e(1)=2.
\end{aligned}\right.\ee
Solving \rf{ex-Pe} by separating variables, we get
$$ P_\e(s)={\e\over \e+1-s},\q\Pi_\e(s)={2\e\over \e+8-8s} \q s\in[0,1]. $$
Let
\begin{align}
\Th_\e &\deq-(R+\e I_m+D^\top P_\e D)^{-1}(B^\top P_\e+D^\top P_\e C+S) \nn\\
       &=-{P_\e\over\e} =-{1\over \e+1-s},\qq s\in[0,1],\nn\\
\ti\Th_\e &\deq-\big(R+\bar R+\e I_m+(D+\bar D)^\top P_\e (D+\bar D)\big)^{-1}\big((B+\bar B)^\top \Pi_\e+(D+\bar D)^\top P_\e (C+\bar C)+S+\bar S\big) \nn\\
  \nn     &=-{2\Pi_\e\over\e} =-{4\over \e+8-8s},\qq s\in[0,1],\\
\label{Th_e}\bar\Th_\e&\deq\ti\Th_\e-\Th_\e,\q v_\e=0.
\end{align}
Then the corresponding closed-loop system \rf{closed-loop-syst-e} can be written as
$$\left\{\begin{aligned}
   dX_\e(s) &= \big\{[\Th_\e(s)-1]X_\e(s)+[1-\Th_\e(s)+2\ti\Th_\e(s)]\dbE[X_\e(s)]\big\}ds  \\
            &\hp{=\ } +\big\{\sqrt{2}X_\e(s)-\sqrt{2}\dbE[X_e(s)]\big\}dW(s), \q s\in[0,1],\\
    X_\e(0) &=\xi.
\end{aligned}\right.$$
It follows that
$$\left\{\begin{aligned}
   d\dbE[X_\e(s)] &= 2\ti\Th_\e(s)\dbE[X_\e(s)]ds  \\
    \dbE[X_\e(0)] &=\dbE[\xi].
\end{aligned}\right.$$
By the variation of constants formula for ODEs, we have
\bel{ex-E-X}
\dbE[X_\e(s)]={\e+8-8s\over\e+8}\dbE[\xi],\q s\in[0,1].
\ee
Applying the variation of constants formula for SDEs, we then get
\begin{align*}
X_\e(s) &= (\e+1-s)\,e^{\sqrt{2}W(s)-2s}\int_0^s {1\over\e+1-r}e^{-[\sqrt{2}W(r)-2r]}[3+2\ti\Th_\e(r)-\Th_\e(r)]\dbE[X_\e(r)]dr \\
    &\hp{=\ } -\sqrt{2}(\e+1-s)\,e^{\sqrt{2}W(s)-2s}\int_0^s {1\over\e+1-r}e^{-[\sqrt{2}W(r)-2r]}\dbE[X_\e(r)]dW(r) \\
&\hp{=\ } +{\e+1-s\over\e+1}\,e^{\sqrt{2}W(s)-2s}\xi, \q s\in[0,1].
\end{align*}
In light of \autoref{thm:SLQ-oloop-kehua}, to prove the open-loop solvability at $(0,\xi)$, it suffices to show the family $\{u_\e(\cd)\}_{\e>0}$ defined by
\begin{eqnarray}\label{u_e}
u_\e(s) &\3n\deq\3n& \Th_\e(s)X_\e(s)+\bar\Th_\e(s)\dbE[X_\e(s)]= \Th_\e(s)X_\e(s)+[\ti\Th_\e(s)-\Th_\e(s)]\dbE[X_\e(s)]\nn\\
        &\3n=\3n&  -e^{\sqrt{2}W(s)-2s}\int_0^s {1\over\e+1-r}e^{-[\sqrt{2}W(r)-2r]}[3+2\ti\Th_\e(r)-\Th_\e(r)]\dbE[X_\e(r)]dr\nn \\
    &\hp{=\ }& +\sqrt{2}e^{\sqrt{2}W(s)-2s}\int_0^s {1\over\e+1-r}e^{-[\sqrt{2}W(r)-2r]}\dbE[X_\e(r)]dW(r)\nn \\
&\hp{=\ }& -{1\over1+\e}e^{\sqrt{2}W(s)-2s}\xi+\Big[{1\over \e+1-s}-{4\over\e+8-8s}\Big]\dbE[X_\e(s)], \q s\in[0,1]
\end{eqnarray}
is bounded in $L^2_\dbF(0,1;\dbR)$. Note that $\xi$ is $\cF_0$-measurable, it is clear to see that
\bel{1}
\dbE\int_0^1\Big|{1\over1+\e}e^{\sqrt{2}W(s)-2s}\xi\Big|^2ds\les\dbE\int_0^1e^{2\sqrt{2}W(s)-4s}|\xi|^2ds=|\xi|^2.
\ee
Next, by \rf{ex-E-X}, we have
\begin{align}
&\dbE\int_0^1\Big|\big[{1\over \e+1-s}-{4\over\e+8-8s}\big]\dbE[X_\e(s)]\Big|^2ds\nn\\
&\q\les\int_0^1\Big|{5\over \e+1-s}\times{\e+8-8s\over\e+8}\dbE[\xi]\Big|^2ds\nn\\
\label{2}&\q\les 25|\dbE[\xi]|^2 \les 25\dbE[|\xi|^2],
\end{align}
and
\begin{align}
&\dbE\int_0^1\Big|e^{\sqrt{2}W(s)-2s}\int_0^s {1\over\e+1-r}e^{-[\sqrt{2}W(r)-2r]}\dbE[X_\e(r)]dW(r)\Big|^2ds\nn\\
&\q\les\Big(\dbE\int_0^1\Big|e^{\sqrt{2}W(s)-2s}\Big|^4ds\Big)^{{1\over2}}\Big(\dbE\int_0^1\Big|\int_0^s {1\over\e+1-r}e^{-[\sqrt{2}W(r)-2r]}\dbE[X_\e(r)]dW(r)\Big|^4ds\Big)^{{1\over2}}\nn\\
\nn&\q\les K\Big(\dbE\int_0^1\int_0^s \Big|{1\over\e+1-r}e^{-[\sqrt{2}W(r)-2r]}\dbE[X_\e(r)]\Big|^4drds\Big)^{{1\over2}}\\
\nn&\q\les K\Big(\dbE\int_0^1\int_0^s \Big|{1\over\e+1-r}e^{-[\sqrt{2}W(r)-2r]}{\e+8-8s\over\e+8}\dbE[\xi]\Big|^4drds\Big)^{{1\over2}}\nn\\
&\q\les K\Big(\dbE\int_0^1\int_0^s e^{-[4\sqrt{2}W(r)-8r]}drds\Big)^{{1\over2}}|\dbE[\xi]|^2\les K\dbE|\xi|^2\label{3},
\end{align}
where $K$ is a generic constant which could be different from line to line.
Similar to \rf{3}, we have
\bel{4}
\dbE\int_0^1\Big|e^{\sqrt{2}W(s)-2s}\int_0^s {1\over\e+1-r}e^{-[\sqrt{2}W(r)-2r]}\dbE[X_\e(r)]dr\Big|^2ds\les K\dbE|\xi|^2.
\ee
Further, we have
\begin{align}
&\dbE\int_0^1\Big|e^{\sqrt{2}W(s)-2s}\int_0^s {1\over\e+1-r}e^{-[\sqrt{2}W(r)-2r]}[2\ti\Th_\e(r)-\Th_\e(r)]\dbE[X_\e(r)]dr\Big|^2ds\nn\\
&\q\les\Big(\dbE\int_0^1\Big|e^{\sqrt{2}W(s)-2s}\Big|^4ds\Big)^{{1\over2}}\Big(\dbE\int_0^1\Big|\int_0^s {1\over\e+1-r}e^{-[\sqrt{2}W(r)-2r]}[2\ti\Th_\e(r)-\Th_\e(r)]\dbE[X_\e(r)]dr\Big|^4ds\Big)^{{1\over2}}\nn\\
\nn&\q\les K\Big(\int_0^1\Big(\int_0^s \Big|{1\over\e+1-r}\dbE[\xi]\Big|dr\Big)^4ds\Big)^{{1\over2}}\Big(\dbE\Big[\sup_{r\in[0,1]}e^{-[4\sqrt{2}W(r)-8r]}\Big]\Big)^{{1\over2}}\\
\nn&\q\les K\Big(\int_0^1[|\ln(\e+1-s)|^4+|\ln(1+\e)|^4]ds\Big)^{{1\over 2}}|E[\xi]|^2\\
\label{5}&\q\les K\Big(\int_0^1[|\ln(1-s)|^4+1]ds\Big)^{{1\over2}}|E[\xi]|^2\les K\dbE|\xi|^2.
\end{align}
Combining \rf{u_e} with the above estimates \rf{1}--\rf{2}--\rf{3}--\rf{4}--\rf{5}, we have
$$
\dbE\int_0^1|u_\e(s)|^2ds\les K\dbE|\xi|^2.
$$
Therefore, $\{u_\e(\cd)\}_{\e>0}$ is bounded in $L^2_\dbF(0,1;\dbR)$. Let $\e\to0$ in \rf{u_e}, we get an open-loop optimal
control:
\begin{align*}
u^*(s) &= -3e^{\sqrt{2}W(s)-2s}\int_0^s e^{-[\sqrt{2}W(r)-2r]}dr\dbE[\xi]+\sqrt{2}e^{\sqrt{2}W(s)-2s}\int_0^s e^{-[\sqrt{2}W(r)-2r]}dW(r)\dbE[\xi]\nn \\
&\hp{=\ } -e^{\sqrt{2}W(s)-2s}\xi+{1\over2}\dbE[\xi], \q s\in[0,1].
\end{align*}

Finally, we let $\e\to0$ in \rf{Th_e}  to get a weak closed-loop optimal strategy $(\Th^*(\cd),\bar\Th^*(\cd))$:
\begin{align*}
  \Th^*(s) &=\lim_{\e\to0}\Th_\e(s) = -{1\over 1-s},  && s\in(0,1),\\
   \bar\Th^*(s) &=\lim_{\e\to0}  \bar\Th_\e(s)=\lim_{\e\to0} [ \ti\Th_\e(s)-\Th_\e(s)] ={1\over 2-2s} , && s\in(0,1).
\end{align*}
We point out that neither $\Th^*(\cd)$ nor $\bar \Th^*(\cd)$ is square-integrable on $(0,1)$. Indeed,
\begin{align*}
   \int_0^1 |\Th^*(s)|^2ds   &= \int_0^1 {1\over(1-s)^2}ds=\i, \\
   \int_0^1 |\bar\Th^*(s)|^2ds   &= \int_0^1 {1\over(2-2s)^2}ds=\i.
\end{align*}
\end{example}

\end{document}